\documentclass[11pt,a4paper]{article}

\usepackage[UKenglish]{babel}
\usepackage[titletoc,title]{appendix}
\usepackage{amsmath,amsfonts,amssymb,bm,amsthm, mathtools}
\usepackage[colorlinks,citecolor=blue,linkcolor=blue,pagebackref]{hyperref}
\usepackage{tikz, pgf, esint}
\usetikzlibrary{math}
\usetikzlibrary{patterns}
\usepackage{graphicx}
\usepackage{enumerate}
\usepackage{xcolor,cancel}
\usepackage{longtable}
\usepackage{enumerate}

\usetikzlibrary{shapes,arrows}
\usetikzlibrary{intersections,shapes.arrows}
\usepackage{caption}

\usepackage{nicefrac}

\numberwithin{equation}{section}
\newtheorem{theorem}{Theorem}[section]
\newtheorem{lemma}[theorem]{Lemma}
\newtheorem{corollary}[theorem]{Corollary}
\newtheorem{proposition}[theorem]{Proposition}
\theoremstyle{definition}

\newtheorem{definition}[theorem]{Definition}
\newtheorem{assumption}[theorem]{Assumption}
\theoremstyle{remark}
\newtheorem{remark}[theorem]{Remark}

\usepackage{nicefrac}

\newcommand{\R}{\mathbb{R}}
\newcommand{\Z}{\mathbb{Z}}
\newcommand{\N}{\mathbb{N}}
\newcommand{\E}{\mathbb{E}}
\newcommand{\PP}{\mathbb{P}}

\newcommand{\cA}{\mathcal{A}}
\newcommand{\cB}{\mathcal{B}}
\newcommand{\cI}{\mathcal{I}}
\newcommand{\cL}{\mathcal{L}}
\newcommand{\cM}{\mathcal{M}}

\newcommand{\dist}{\operatorname{dist}}

\renewcommand{\tilde}{\widetilde}
\renewcommand{\epsilon}{\varepsilon}

\usepackage[T1]{fontenc}
\usepackage[top=1in, bottom=1in, left=1in, right=1in]{geometry}
\usepackage{fancyhdr}
\usepackage{relsize}

\usepackage{orcidlink}

\allowdisplaybreaks

\usepackage{pifont}
\renewcommand*{\backrefalt}[4]{%
\ifcase #1 %
No citations%
\or
\ding{43}~p.~#2%
\else
\ding{43}~pp.~#2%
\fi}

\usepackage[normalem]{ulem}

\begin{document}

\title{Anderson localization in high-contrast media\\ with random spherical inclusions
}

\author{
Matteo Capoferri\,\orcidlink{0000-0001-6226-1407}\thanks{Dipartimento di Matematica ``Federigo Enriques'', Università degli Studi di Milano, Via C.~Saldini 50, 20133 Milano Italy \emph{and} Maxwell Institute for Mathematical Sciences, Edinburgh \&
Department of Mathematics,
Heriot-Watt University,
Edinburgh EH14 4AS,
UK;
\text{matteo.capoferri@unimi.it},
\url{https://mcapoferri.com}.
}
\and
Matthias T\"aufer\,\orcidlink{0000-0001-8473-2310}\thanks{
CERAMATHS, Université Polytechnique Hauts-de-France,
Le Mont Houy 59313 Valenciennes Cedex 09,
France;
\text{Matthias.Taufer@uphf.fr}.
}
}


\date{2 December 2025 \vspace{-.75cm}}

\maketitle
\begin{abstract}

 We study spectral properties of partial differential operators modelling composite materials with highly contrasting constituents, comprised of soft spherical inclusions with random radii dispersed in a stiff matrix.
 Such operators have recently attracted significant interest from the research community, including in the context of stochastic homogenization. 
 In particular, it has been proved that the spectrum of these operators may feature a band-gap structure in the regime where heterogeneities take place on a sufficiently small scale. However, the nature of the limiting (as the small scale tends to zero) spectrum in the above setting 
 is non-classical and not completely understood.
 In this paper we prove for the first time that Anderson localization occurs near band edges, thus shedding light on the limiting spectral behaviour.
 Our results rely on recent nontrivial advancements in quantitative unique continuation for PDEs, in combination with assumptions on the model that are standard in the Anderson localization literature, and which we plan to relax in future works.

\


\

{\bf 2020 MSC classes: }
Primary 
47B80, 
74A40. 
Secondary 
35B27, 
35R60, 
74S25, 
82B44. 
\end{abstract}

\tableofcontents

\allowdisplaybreaks

\section{Introduction}
\label{sec:Introduction}

The goal of this paper is to make a first step towards a rigorous mathematical understanding of Anderson localization in the context of high-contrast random media.

Over the years, high-contrast media, that is, composite materials whose constituents possess very different (or ``highly contrasting'') physical properties, have attracted the attention of both mathematicians and applied scientists, because, unlike more traditional composites, they possess special features allowing for effective manipulation of propagating waves.
Indeed, in this class of materials, often modelled as a collection of ``soft'' inclusions dispersed in a ``stiff'' matrix, the interplay between the characteristic size $\epsilon$ of heterogeneities and the ratio $\epsilon^2$ between the diffusion coefficients of inclusions and matrix brings about resonances at the microscopic scale which yield a spectrum with band-gap structure. This particular relative scaling is often referred to as \emph{double-porosity} resonant regime.

From a mathematical perspective, the traditional approach to study composite materials with small heterogeneities is via homogenisation theory.
An immediate mathematical challange one encounters when trying to do so for high-contrast media is that the family of operators $\mathcal{A}^\epsilon:=- \operatorname{div}\,a^\epsilon\, \nabla$  ($a^\epsilon$ being the diffusion coefficients) describing them is not uniformly elliptic as $\epsilon\to 0$, ultimately leading to a loss of compactness in the limit. Hence, classical homogenisation does not work.

In the high-contrast \emph{periodic} setting --- that is, when inclusions are arranged periodically in space ---
pioneering works by Allaire first \cite{Allaire-92} and by Zhikov shortly thereafter \cite{Zhikov-00,Zhikov-04} demonstrated that one can still define a limiting ``homogenised'' operator in the sense of the so-called \emph{two-scale convergence}. This limiting operator possesses a two-scale nature, in that it accounts for both macroscopic and microscopic properties of the material, coupled in a particular way. Remarkably, Zhikov showed that the two-scale limit yields convergence of spectra in the sense of Hausdorff, thus making this technique well-suited for the spectral analysis of high-contrast media.

From the point of view of applications, requiring perfect periodicity is often too strong a constraint, and one would like to allow for some randomness in the coefficients, the geometry of the inclusions, or both. 
Whilst there exists extensive literature on high-contrast periodic media, see, e.g., \cite{CherednichenkoSZ-06, CherednichenkoC-15, CherednichenkoC-18, CherednichenkoVZ-25, CooperKS-23, CherednichenkoKVZ-25}, a systematic study of the random setting was initiated only recently, so that the literature is limited and much less is known. It should be mentioned that this is not accidental: moving from periodic to random operators involves a significant step change in both conceptual and technical challenges.

Building upon earlier work by Zhikov and Piatnitski \cite{ZhikovP-06}, Cherdantsev, Cherednichenko and Vel\v{c}i\'{c} in \cite{CherdantsevCV-18} proved that on bounded domains one has stochastic two-scale resolvent convergence as well as Hausdorff spectral convergence for high-contrast operators with random coefficients. In the whole space $\R^d$ the picture is rather different: indeed, a few years later the same set of authors showed in \cite{CherdantsevCV-21} that, although one still has two-scale convergence, in general the spectrum of the limiting operator is a proper subset of the limiting spectrum $\lim_{\epsilon\to 0}\sigma(\mathcal{A}^\epsilon)$. That is, the operators $\mathcal{A}^\epsilon$ shortly before the limit possess spectrum that `is not seen' by the limiting operator. Nevertheless, by examining a range of physically meaningful examples, they showed that the limiting spectrum (and, thus, the spectrum of the homogenised operator) can exhibit infinitely many gaps.

So as not to overload the paper and disperse the focus, we refrain here from discussing subsequent literature on high-contrast stochastic homogenisation (mostly qualitative so far, with a couple of very recent quantitative results), limiting ourselves to listing those references closest to the setting of our paper, for the benefit of the reader: \cite{CapoferriCV-23,CapoferriCV-25, BellaCCV-25,BonhommeDG-25,AmazianePZ-23}.

\

Informal, non-rigorous examination of specific examples where Hausdorff convergence of spectra fails in the stochastic setting seems to suggest that this type of materials may exhibit dense pure point spectrum with well localized eigenfunctions near the edges of the spectral bands --- a phenomenon known under the name of \emph{Anderson localization}.

{
The theory of Anderson localization goes back to Anderson's pioneering work~\cite{Anderson-58} who identified it as \emph{absence of diffusion} in certain random lattice operators.
Since then, the mathematical treatment of Anderson localization has had a long history with numerous contributions. 
We highlight~\cite{GoldsheidMP-77} as the first mathematically rigorous work proving localization in continuum models in one dimension, and~\cite{HoldenM-84} for the first proof of Anderson localization in continuum models in $\R^d$ at the bottom of the spectrum.
Furthermore, we should like to mention~\cite{FroehlichS-84} where the most common technique to prove localization, that is, \emph{multiscale analysis}, has been developed (albeit for the special case of discrete Laplacians), and refer the reader to the monographs~\cite{CarmonaL-90, Stollmann-01, Veselic-08, AizenmanW-15} for more details and references.

The bulk of works on Anderson localization for operators in $L^2(\R^d)$ treat the \emph{alloy-type} or \emph{continuum Anderson model}
\begin{equation}
    \label{eq:alloy_type}    
    H_{\omega}
    =
    - \Delta
    +
    V_\omega
    \quad
    \text{where}
    \quad
    V_\omega(x)
    =
    \sum_{j \in \Z^d}
    \omega_j \,,
    u(x-j)
\end{equation}
where $\omega = (\omega_j)_{j \in \Z^d}$ is a family of (usually bounded) independent and identically distributed random variables and $u$ is a (usually bounded, nonnegative, and compactly supported) profile function, modeling the influence of a single atom.
The advantage of this model is that it gives rise to linear perturbation theory: changes to the elementary random variables correspond to linear perturbations of the operator $H_\omega$.

There are some exceptions to this in which localization has been proved, the most noteworthy being the \emph{Poisson model}~\cite{GerminetHK-07}
\[
    V_\omega (x)
    =
    \sum_{y \in \mathcal{P}(\omega)}
    u(x-y)\,,
\]
where $\mathcal{P} \subset \R^d$ is a Poisson point process, the~\emph{random displacement model}~\cite{KloppLNS-12} 
\[
    V_\omega(x)
    =
    \sum_{j \in \Z^d}
    u(x - j - \eta_j)\,,
\]
where the family $(\eta_j)_{j \in \R^d}$ is a family of independent and identically distributed, $\R^d$-valued \emph{random displacements}, the \emph{random breather model}~\cite{TaeuferV-15, NakicTTV-18}
\[
    V_\omega(x)
    =
    \sum_{j \in \Z^d}
    \mathbf{1}_{B_{\omega_j}(j)}\,,
\]
where $B_{\omega_j}(j)$ are balls, centered at $j$ with independent and identically distributed \emph{random radii} $(\omega_j)_{j \in \Z^d}$ --- amongst others, including \emph{Gaussian random potentials}~\cite{FischerLM-00}.
A reasonably systematic understanding of random operators where the randomness does not enter linearly has only recently been developed~\cite{TaeuferV-15, NakicTTV-18} and has required significant progress in quantitative unique continuation principles~\cite{BourgainK-05, RojasMolinaV-13,Klein-13}.

All the models discussed so far have in common is that the randomness affects the \emph{potential}, that is the zeroth-order term of a second-order elliptic differential operator.
Proofs of Anderson localization in models where the randomness affects higher-order terms of the operator are even rarer. 
For perturbations affecting first order terms of the differential operators, the most notable ones are examples of localization by random magnetic fields~\cite{ErdosH-12}, as well as some examples of abstract perturbations under technical positivity assumptions~\cite{BorisovTV-21}.

As for randomness affecting the second-order term of a differential operator, the analysis is more challenging still, and necessitates a more delicate spectral analysis. 
The only examples of proofs of Anderson localization in $\R^d$ we are aware of are~\cite{Stollmann-98, CombesHT-99} where the randomness again enters linearly in an Alloy-type fashion, i.e.
\begin{equation}
    \label{eq:div_grad_alloy_type}
    H_\omega 
= 
- \operatorname{div} a_\omega(x) \nabla
\quad
\text{with}
\quad
a_\omega(x)
=
\mathbf{1}
+
\sum_{j \in \mathbb{Z}^d}
\omega_j u(x-j).
\end{equation}

There is an additional difficulty in second-order random operators in $\R^d$: localization is not expected near the bottom of the spectrum, since it is not a so-called \emph{fluctuation boundary}~\cite{Stollmann-98}, but rather near the edges of spectral band gaps.
Anderson localization has indeed been proved near higher spectral band edges but only in cases where the randomness enters linearly, that is either in Alloy-type models of the form~\eqref{eq:alloy_type}
\[
H_\omega 
=
H_0 + V_\omega,
\quad
V_\omega(x)
=
\sum_{j \in \Z^d}
\omega_j
u(x-j)\,,
\]
where $H_0$ is an operator with a band-gap structure --- e.g., of the form $- \Delta + V_{\operatorname{per}}$ for a periodic potential $V_{\operatorname{per}}$, see~\cite{BarbarouxCH-97, KirschSS-98, Klopp-99, Veselic-02, KloppW-02}, or the Landau operator~\cite{CombesH-96} ---
or in divergence-type operators where the conductivity is of alloy-type form~\eqref{eq:div_grad_alloy_type}, see~\cite{Stollmann-98,CombesHT-99}.

\

Indeed, since spectral perturbation theory is inherently easier near the bottom of the spectrum, the treatment of Anderson localization near higher band edges has required more technical effort. 
We are aware of three main approaches in the literature. 

The first, most classic approach relies on a condition of \emph{thinness of random variables near their extrema}. 
This has been a standing assumption in the first proofs of localization near internal band edges~\cite{BarbarouxCH-97, KirschSS-98, Stollmann-98}, and is still being adopted in recent literature~\cite{BarbarouxCZ-19}. We shall also rely on this assumption.

The second method to deal with band edge localization has arisen in the 2000s and is based on \emph{Floquet theory}.
It leverages on the periodicity of the background operator $H_0$, thus allowing to refine the perturbation theory of the random operator $H_\omega$ around spectral band edges.
While this has led to significant progress --- and has in particular allowed to drop the assumption of thinness of random variables near their extreme in Anderson localization for band-edge localization in the alloy type model --- the proofs rely on the so-called ``regularity of Floquet edges'', which is still only conjectured to be universal (i.e., counterexamples are supposed to be rare).
Furthermore, Floquet-theory-based proofs of Anderson localization near band edges have been limited to the Alloy-type model, see~\cite{Veselic-02} for a discussion.

The third method, proposed in~\cite{SeelmannT-20} (see also~\cite{RojasMolina-Thesis} where similar ideas are already used at the bottom of the spectrum), builds upon years of progress on quantitative unique continuation, initiated in~\cite{BourgainK-05} and further developed in~\cite{RojasMolinaV-13, Klein-13, NakicTTV-15, NakicTTV-18}. This method uses \emph{quantitative unique continuation estimates} with sufficient information on geometric dependence in combination with large deviation estimates.
The latter has allowed to universally prove Anderson localization near band edges for alloy-type models, but also to show band-edge localization for other random operators where the disorder only affects the potential, such as the random breather model, cf.~\cite[Remark~3.3(4)]{SeelmannT-20}.

\

In this paper, which fits into the above broader landscape, we examine the issue of localization in a simplified model of high contrast random media, with two main objectives in mind: (i) to show that high-contrast random media \emph{do} indeed exhibit localization and (ii) to do so in a way that is accessible to different research communities which are unfamiliar with either multiscale analysis tools (in the sense of \cite{FroehlichS-84}) or with high-contrast stochastic homogenisation. Whilst there has been recent and interesting progress in the understanding of localization in sub-wavelength resonators (see, e.g., \cite{AmmariDH-24, AmmariHR-25}), Anderson localization has not --- to the best of our knowledge --- yet been studied or observed in the setting of the current paper, which builds upon the analysis of the aforementioned~\cite{CherdantsevCV-21,CapoferriCV-23}.  A rigorous understanding of localization in such framework is an exciting prospect, because potentially conducive to applications, and because unravelling the connection between mathematical techniques and the underlying physical mechanism that brings about localization (micro-resonances) may provide fresh insight into the former and lead to further ``purely mathematical'' advances.

\

In order to close the arguments, we shall need to impose here additional assumptions on our stochastic model, see Assumptions~\ref{assumption:bounded_density} and~\ref{assumption:thinness_density}, which we plan to relax in future works, upon developing novel analytical tools required to tackle the mathematical obstacles offered by the general framework. Indeed, there are a number of nontrivial challenges in attempting to undertake the above programme, mostly stemming from the nature of the operator under analysis which differs in significant ways what has traditionally been studied in the context of Anderson localization. Namely:
\begin{enumerate}[(i)]
    \item 
    The randomness only affects the second-order, that is the principal part of a random operator.
    This has been treated but only with linear influence of elementary random variables~\cite{Stollmann-98, FigotinK-97}.
    \item 
    Randomness enters in non-linear way, that is in the radii of inclusions.
    Such a mechanism of randomness has only recently been treated for a random \emph{potential}~\cite{TaeuferV-15}, which already had required a breakthrough in quantitative unique continuation~\cite{NakicTTV-18}.
    \item 
    We no longer work at the bottom of the spectrum but at internal band edges where the perturbation theory becomes more complicated.
    \item 
    In the classical high-contrast case, we want to be able to deal with eventually, the coefficients are non-Lipschitz, so that even the above mentioned recent advancements will not be sufficient.
\end{enumerate}
This justifies the simplifications introduced in this first paper on the subject. 
At a conceptual level, the key ingredient in our proof of Anderson localization are recent scale-free quantitative unique continuation estimates for the gradient with explicit dependence on the geometry, obtained in~\cite{Dicke-21, DickeV-23}, which allow us to prove a Wegner and an initial scale estimate for model at hand.
They also appear to constitute, at present, the main bottleneck for generalisations of our results on localization, e.g. to non-Lipschitz continuous diffusivity functions and more singular random variables.

\

For the sake of completeness, let us finally point out that there have been other important developments in Anderson localization in the last decades which do not directly pertain to this paper, such as the proof of localization in the discrete Anderson-Bernoulli model in $\Z^2$~\cite{DingS-19}, the development of the fractional moment method for localization~\cite{AizenmanENSS-06}, the proof of Poisson statistics of eigenvalues~\cite{DietleinE-20}, as well as the connection between Anderson localization and the torsional rigidity or landscape function~\cite{FilocheM-12}.
Overall, despite many years have passed since Anderson's seminal paper, Anderson localization is still a very active field of cutting-edge mathematical research, which testifies to the highly nontrivial challenges in rigorously capturing this very elusive phenomenon.

\subsection*{Structure of the paper}
\addcontentsline{toc}{subsection}{Structure of the paper}

Our paper is structured as follows.

\

In Section~\ref{sec:Statement of the problem} we introduce the mathematical framework of our work and rigorously formulate the problem we set out to address.
Section~\ref{sec:Main results} summarises the main results of the paper in the form of five theorems, the proofs of which are postponed until later sections.
In Section~\ref{sec:Gaps in the spectrum} we demonstrate that our setting is meaningful, namely, that the spectrum of our family of operators has a band-gap structure.
The remaining three sections are concerned with the proofs of our main results: Section~\ref{Wegner estimate} focusses on the Wegner estimate, Section~\ref{Initial scale estimate under the strong assumption} on the initial scale estimate, whereas in Section~\ref{sec:MSA} we carry out the multiscale analysis and prove our main result --- Anderson localization (in fact, we prove a bit more than that).

For the convenience of the readers, the paper is complemented by a list of notation, which can be found at the end of Section~\ref{sec:Introduction}.

\subsection*{List of notation}
\addcontentsline{toc}{subsection}{List of notation}
\begin{longtable}{l l}
\hline
\\ [-1em]
\multicolumn{1}{c}{\textbf{Symbol}} & 
  \multicolumn{1}{c}{\textbf{Description}} \\ \\ [-1em]
 \hline \hline \\ [-1em]
$\langle \,\cdot\,,\,\cdot\,\rangle$ (resp.~$\|\,\cdot\,\|$) & Inner product (resp.~norm) in $L^2$\\ \\ [-1em]
$a_\omega^\epsilon$ (resp.~$\hat a_\omega^\epsilon$) & Diffusion coefficients \eqref{a_epsilon} (resp.~\eqref{a hat})
\\ \\ [-1em]
$\cA_\omega^\epsilon$, $\cA_{\omega,L,x}^\epsilon$, $\cA_{\omega,L}^\epsilon$ & Definition~\ref{def:A_epsilon}  \\ \\ [-1em]
$\hat{\cA}_\omega^\epsilon$ & High-contrast operator without boundary layer, defined by \eqref{eq:definition_A_hat_omega_epsilon_weak}  \\ \\ [-1em]
$\alpha_\epsilon$ & Constant \eqref{eq:alpha_epsilon}
\\ \\ [-1em]
$B_r$ (resp.~$B_r(x)$) & Open ball of radius $r>0$ centred at $0$ (resp.~$x$) \\ \\ [-1em]
$d$ & Space dimension \\ \\ [-1em]
$\mathcal{D}(\cA)$ & Domain of the opeartor $\cA$ \\ \\ [-1em]
$\mathbf{e}$ & Unit sequence \eqref{eq:notation_e_and_sL}\\ \\ [-1em]
$E_0$ & Lower band edge \eqref{eq:notation_e_and_sL}\\ \\ [-1em]
$\E[\,\cdot\,]$  & Expectation \\ \\ [-1em]
$\gamma$  & Parameter $\ge 2$ determining thickness of boundary layer, see~\eqref{L_omega}\\ \\ [-1em]
$\mathrm{I}$ & Identity matrix \\ \\ [-1em]
$\cI_\omega^\epsilon$ & Collection of (spherical) inclusions~\eqref{I_omega} \\ \\ [-1em]
$\cL_\omega^\epsilon$ & Collection of boundary layers of inclusions~\eqref{L_omega} \\ \\ [-1em]
$\Lambda_L$ (resp.~$\Lambda_L(x)$) & Open box $(0,L)^d$ (resp.~open box $x+(0,L)^d$)\\ \\ [-1em]
$\cM_\omega^\epsilon$ & Matrix~\eqref{M_omega} \\ \\ [-1em]
$\mu$ & Probability measure \\ \\ [-1em]
$\eta_\mu$ & Probability density \\ \\ [-1em]
$s_c$ & Shift vector \eqref{eq:notation_e_and_sL}, $c\in\R$ \\ \\ [-1em]
$\sigma(\cA)$ & Spectrum of the operator $\cA$ \\ \\ [-1em]
$\mathrm{tr}[\,\cdot\,]$ & Operator trace \\ \\ [-1em]
$\chi_\Lambda$ & Characteristic function of the set $\Lambda$\\ \\ [-1em]
$\omega$ & Collection of radii of inclusions, an element of the probability space \\ \\ [-1em] 
$[\omega_-,\omega_+]$ & Support of the probability density, subject to the constraint~\eqref{eq:bounds on omega minus and plus} \\ \\ [-1em]
$(\Omega, \mathcal{F}, \PP)$ & Probability space --- see~\eqref{Omega}--\eqref{P} \\ \\ [-1em]
\hline
\end{longtable}

\section{Statement of the problem}
\label{sec:Statement of the problem}

We work in Euclidean space $\R^d$, where $d \geq 2$.
We denote by $B_r(x)$ the open ball of radius $r>0$ centred at $x\in \R^d$ and by $\Lambda_L(x):=x+(0,L)^d$ the open box $(0,L)^d$ of side length $L$ shifted by $x$. When $x=0$ we simply write $B_r$ and $\Lambda_L$ for $B_r(0)$ and $\Lambda_L(0)$, respectively. Furthermore, we denote by $\chi_S$ the characteristic function of the set $S\subset \R^d$.
For later notational convenience, we introduce the following definition.
\begin{definition}
    \label{def:equidistributed}
    Let $G, \delta > 0$.
    A sequence of points $X = \{x_z\}_{z \in \Z^d}$ is \emph{$(G,\delta)$-equidistributed} if 
    \[
    B_\delta(x_z) \subset \Lambda_G(G z)
    \quad
    \text{for all $z \in \Z^d$}.
    \]
\end{definition}
Note that $(G, \delta)$-equidistributed sets are a special class of so-called \emph{thick} or \emph{relatively dense} sets~\cite{EgidiV-18} and seem to be the appropriate class for unique quantitative continuation principles for Schr\"odinger operators~\cite{NakicTTV-18, TaeuferT-18, DickeV-23}.

The symbol $\sigma(\cA)$ denotes the spectrum of an operator $\cA$.
If $\cA$ is self-adjoint and lower semibounded with purely discrete spectrum, we denote by $\lambda_k(\cA)$ its $k$-th eigenvalue, with account of multiplicity.

We shall also adopt the notation 
\begin{equation}
\label{eq:notation_e_and_sL}
    \mathbf{e}:=\{1\}^{\Z^d}, 
    \quad
    \text{and}
    \quad 
    s_c
    :=
    \left(
    \frac c2, \cdots, \frac c2
    \right)
    \in \R^d\,.
\end{equation}
In particular, $x + s_c$ is the centre point of $\Lambda_L(x)$.

Moreover, $L^2(S)$, $H^1(S)$, and $H^1_0(S)$, with $S\subseteq \R^d$, denote the usual Lebesgue and Sobolev spaces. We denote by $\langle \,\cdot\,,\,\cdot\,\rangle$ and $\|\,\cdot\,\|$ the inner product and norm in $L^2$, respectively.

\

For 
\begin{equation}
    \label{eq:bounds on omega minus and plus}
    0<\omega_-<\omega_+<\frac14\,,
\end{equation}
let $\mu$ be a probability measure supported in $[\omega_-,\omega_+]$ with bounded probability density $\eta_\mu$ and such that $\omega_- = \min \operatorname{supp} \mu$. We define a probability space $(\Omega, \Sigma, \PP)$ as
\begin{equation}
\label{Omega}
\Omega:=[0,+\infty)^{\Z^d}\,,
\end{equation}
\begin{equation}
\label{Sigma}
\mathcal{F}:= \bigtimes_{z\in \Z^d} \cB([0,+\infty))\,,
\end{equation}
\begin{equation}
\label{P}
\PP:=\bigotimes_{z\in \Z^d} \eta_\mu\,,
\end{equation}
where $\cB([0,+\infty))$ denotes the Borel $\sigma$-algebra on $[0,+\infty)$.

\

For $\omega=(\omega_z)\in \left(0, \frac{1}{4} \right)^{\Z^d}$, $\epsilon\in (0,1)$, and $\gamma\ge2$, define
\begin{equation}
\label{I_omega}
\cI_\omega^\epsilon:=\bigcup_{z\in \epsilon\Z^d} B_{\epsilon\omega_z}(z+s_\epsilon)\,, \qquad \cI_\omega:=\cI_\omega^1,
\end{equation}
\begin{equation}
\label{L_omega}
\cL_\omega^\epsilon:=\bigcup_{z\in \epsilon\Z^d} B_{\epsilon\omega_z+\nicefrac{\epsilon^\gamma}4}(z+s_\epsilon)\setminus B_{\epsilon\omega_z}(z+s_\epsilon)\,,
\end{equation}
\begin{equation}
\label{M_omega}
\cM_\omega^\epsilon:=\R^d \setminus (\cI_\omega^\epsilon \cup \cL_\omega^\epsilon)\,.
\end{equation}
The set $\cI_\omega^\epsilon$ models a collection of spherical \emph{inclusions} with centres in $\epsilon\Z^d+s_\epsilon$, of  characteristic size of order $\epsilon$ and radii ranging from $0$ to $\epsilon/4$, the set $\cM_\omega^\epsilon$ models a \emph{material matrix}, and $\cL_\omega^\epsilon$ is a boundary \emph{layer} connecting the two. Note that
$B_{\epsilon\omega_z+\frac{\epsilon^\gamma}4}\subset B_\frac\epsilon2$ for all $\epsilon\in(0,1)$ and $\omega\in [\omega_-,\omega_+]^{\Z^d}$, which implies that inclusions and their boundary layers are fully contained in the interior of the boxes $\Lambda_{\epsilon}(z)$, $z\in \epsilon \Z^d$, see~Figure~\ref{fig:I_L_M}.

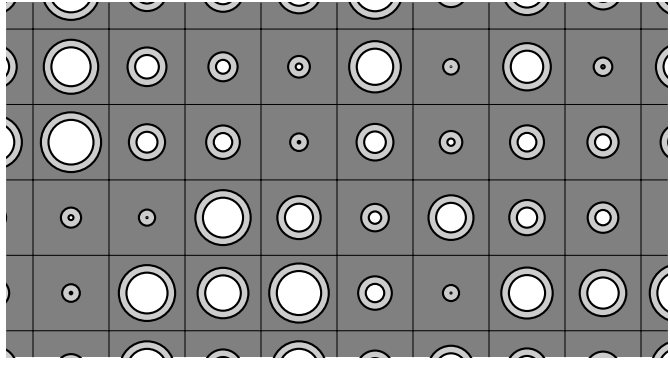
\begin{figure}[ht!]
    \centering

\begin{tikzpicture}
    \begin{scope}
    \clip (-.35,-.35) rectangle (8.35,4.35);
    \fill[black!50] (-.35,-.35) rectangle (8.35,4.35);
    \draw[very thin] (-.35,-.35) grid (8.35,4.35);
    \foreach \x in {-1,...,8}
        {
        \foreach \y in {-1,...,4}
        {
        \pgfmathsetmacro{\w}{rand*0.15 + 0.15};
        \draw[thick, fill = black!20] (\x+.5,\y+.5) circle (\w cm + .1cm); 
        
        \draw[thick, fill = white] (\x+.5,\y+.5) circle (\w cm);

        }
        }
    \end{scope}
    
\end{tikzpicture}

\caption{The sets $\cI_\omega^\epsilon$ (white), $\cL_\omega^\epsilon$ (light grey) and $\cM_\omega^\epsilon$ (dark grey) in two dimensions (left). On the right-hand side, the same sets for a Delone-configuration.}
\label{fig:I_L_M}
\end{figure}

\

The geometric structure of our heterogeneous medium is described by the above sets, whereas its material properties are described by the diffusion coefficients
\begin{equation}
\label{a_epsilon}
a_\omega^\epsilon(x):=
\begin{cases}
1 & x\in \cM_\omega^\epsilon\,,
\\
\epsilon^2 &  x\in \cI_\omega^\epsilon\,,
\\
1-(1-\epsilon^2) \frac{\dist(x,\cM_\omega^\epsilon)}{\nicefrac{\epsilon^\gamma}4} &  x\in \cL_\omega^\epsilon\,.
\end{cases}
\end{equation}

More precisely, the coefficients \eqref{a_epsilon} describe a \emph{high-contrast} composite, in that we have soft inclusions with diffusivity $\epsilon^2$, dispersed in a stiff matrix of diffusivity $1$, with a linear transition between the two phases taking place in a boundary layer of thickness $\epsilon^\gamma/4$. 
The parameter $\epsilon > 0$ is a model parameter which should be thought of as a fixed number, close to zero.
Indeed, it arises from considerations on homogenization and choosing $\epsilon$ sufficiently small will ensure that the operators we consider have gapped spectrum, see Section~\ref{sec:Gaps in the spectrum}.

It is not hard to see that the coefficients satisfy the following properties.
\begin{enumerate}[(a)]
    \item We have 
    \begin{equation}
        \label{eq:coeff property 1}
        \epsilon^2\le a_{\omega}^\epsilon\le 1\,.
    \end{equation}
    \item 
    Let 
    \begin{equation}
        \label{eq:alpha_epsilon}
        \alpha_\epsilon:=\frac{2(1-\epsilon^2)}{\epsilon^{\gamma-1}}.
    \end{equation}
    Then there exists an $(\epsilon,\frac{\epsilon^\gamma}{10})$-equidistributed sequence $X=\{x_z\}_{z\in\epsilon \Z^d}$, see Definition~\eqref{def:equidistributed}, such that 
    \begin{equation}
    \label{eq:coeff_property_2a}
        a_\omega^\epsilon-a_{\omega+s \mathbf{e}}^\epsilon
        \ge \alpha_\epsilon \,s\sum_{z\in \epsilon \Z^d}\chi_{B_{\frac{\epsilon^\gamma}{10} s}(x_z)}
    \end{equation}
     for all $0\le s\le s_0:=\frac14 \epsilon^{\gamma-1}$.
    In particular, we have
    \begin{equation*}
         \label{eq:coeff_property_2b}
         \|a_\omega^\epsilon-a_{\omega+s \mathbf{e}}^\epsilon\|_\infty \ge \alpha_\epsilon\,s\, \qquad \forall s\in[0,s_0].
    \end{equation*}
    \item The random function $a_\omega^\epsilon$ is uniformly Lipschitz continuous with Lipschitz constant
    \begin{equation*}
    \label{Lipischitz constant}
    C_L=C_L(\epsilon)\le \frac{4}{\epsilon^\gamma}
    \end{equation*}
    uniformly in $\omega$.
\end{enumerate}

\begin{definition}
\label{def:A_epsilon}
    We define $\cA_\omega^\epsilon$ to be the self-adjoint linear operator in $L^2(\R^d)$ associated with the bilinear form
    \begin{equation*}
        \label{eq:definition_A_omega_epsilon_weak}
        \int_{\R^d} a_\omega^\epsilon \nabla u \cdot \nabla v\,, \qquad u,v\in H^1(\R^d)\,.
    \end{equation*}
    Similarly, for $L > 0$ and $x \in \R^d$ we define the operator $\cA_{\omega,L,x}^\epsilon$ to be the self-adjoint linear operator in $L^2(\Lambda_L(x))$ associated with the bilinear form
    \begin{equation*}
        \int_{\Lambda_L(x)} a_\omega^\epsilon \nabla u \cdot \nabla v\,, \qquad u,v\in H^1_0(\Lambda_L(x))\,.
    \end{equation*}
In other words, $\cA_{\omega,L,x}^\epsilon$ is the restriction of $\cA_\omega^\epsilon$ to $\Lambda_L(x)$ with Dirichlet boundary conditions.
\end{definition}

Of course, the operator $\cA_\omega^\epsilon$ can be understood as the operator given by $\operatorname{div} (a_\omega^\epsilon \nabla)$.
The goal of this paper is to show that, almost surely and for sufficiently small $\epsilon > 0$, the operator $\cA_\omega^\epsilon$ has gapped spectrum and exhibits Anderson localization in the vicinity of spectral band edges.

\

In order to be able to prove Anderson localization, we will need to impose additional assumptions on the probability measure $\mu$. We collect these assumptions below; they will appear, in different combinations, in the statements of our main results.

\begin{assumption}[Bounded density]
\label{assumption:bounded_density}
    The measure $\mu$ has a bounded density $\eta_\mu$ with respect to the Lebesgue measure.
\end{assumption}

\begin{assumption}[Thinness near the maximum]
    \label{assumption:thinness_density}
    There exists $\kappa > 0$ such that 
    \[
    \mu ([\omega_+ - s, \omega_+])
    \leq
    s^{\kappa}
    \]
    for all sufficiently small $s > 0$.
\end{assumption}

\begin{remark}
Below, in the proof of Theorem~\ref{thm:ISE_strong}, we will require Assumption~\ref{assumption:thinness_density}  to hold for some (implicit) $\kappa > 0$.
However, we emphasise that this is not an empty assumption. 
For instance, if the density $\eta_\mu$ is \emph{exponentially thin} near the maximum of its support then Assumption~\ref{assumption:thinness_density} will hold for all $\kappa > 0$.
\end{remark}

\begin{remark}
    Note that property~\eqref{eq:coeff_property_2a} tells us that the family of differential operators $\{\cA_\omega\}$ is ``antitone in the randomness''. Thus, Assumption~\ref{assumption:thinness_density} will influence the density of eigenvalues near the \emph{minima} of spectral bands. 
    With obvious modifications, \emph{maxima} of spectral bands can be considered as well and the results of this article hold there analogously. 
\end{remark}

\begin{assumption}
    \label{assumption:epsilon_and_omega}
    The parameters $\epsilon > 0$ and $0 < \omega_- < \omega_+ < \frac{1}{4}$ are chosen in such a way that $\mathcal{A}_\omega^\epsilon$ has at least one gap in its spectrum, almost surely.
\end{assumption}

Assumption~\eqref{assumption:epsilon_and_omega} can always be achieved for sufficiently small $\epsilon$ and sufficiently small range of radii $\omega_+ - \omega_-$.
This will be the subject of Section~\ref{sec:Gaps in the spectrum}.

\

We conclude this section by introducing some notation which will be useful further on.

\begin{definition}[Lower and upper band edges] 
	We call $E_0 \in \R$ a \emph{lower (upper) band edge} of $\sigma(\cA_{\omega,L}^\epsilon)$ if $E_0$ is the minimum (maximum) of a connected component of $\sigma(\cA_{\omega,L}^\epsilon)$, $\PP$-almost surely.
\end{definition}
In particular, if $E_0$ is a lower band edge of $\sigma(\cA_{\omega,L}^\epsilon)$, there is $E_- < E_0$ such that $\PP$-almost surely
\begin{equation*}
    \label{eq:ISE_strong_eq1}
         \sigma (\cA_{\omega}^\epsilon) \cap [E_-, E_0) = \emptyset\,.
    \end{equation*}

\section{Main results}
\label{sec:Main results}

The main result of our paper is the following theorem.

\begin{theorem}[Anderson localization]
\label{thm:Anderson_localization}
Suppose that Assumptions~\ref{assumption:bounded_density} and~\ref{assumption:epsilon_and_omega} hold, and let $E_0 > 0$ be a lower band edge of $\sigma(\cA_{\omega}^\epsilon)$.
Then there exist $\kappa>0$ and $E_+ > E_0$ such that, if Assumption~\ref{assumption:thinness_density} with exponent $\kappa$ is satisfied, the spectrum of $\cA_\omega^\epsilon$ in $[E_0, E_+]$ is pure point with exponentially decaying eigenfunctions $\PP$-almost surely.  
\end{theorem}

This will be achieved via a sequence of intermediate steps, captured by the following key results.

\

The first two ingredients are a Wegner estimate and an initial scale estimate, which concern spectral properties of our operator restricted to boxes.
The former provides an estimate from above for the expected number of eigenvalues thereof in a given spectral window; the latter ensures that for sufficiently large boxes eigenvalues near the bottom of the spectral bands are rare.

\begin{theorem}[Wegner estimate]
    \label{thm:Wegner}
    Let $\epsilon > 0$. Under Assumption~\ref{assumption:bounded_density}, for all $0 < E_- < E_+ < + \infty$, there exist constants $\delta_{\max}>0$, $L_{\min} \in \epsilon 
    \N$,  $C_1 > 0$ and $C_2 > 1$, depending only on $\epsilon, \gamma, E_-, E_+, \lVert \eta_\mu \rVert_{\infty}$, such that for all $0 < \delta \leq \delta_{\max}$, all $E \in [E_-, E_+]$ such that $[E-3\delta, E+3\delta]\subset [E_-,E_+]$, and all $L \in \epsilon 
    \N$ with $L \geq L_{\min}$, we have
    \begin{equation}
        \label{eq:Wegner}
        \E
        \left[
            \operatorname{Tr}
            \left[
               \chi_{[E-\delta, E+\delta]}
                (
                \cA_{\omega, L}^\epsilon
                )
            \right]
        \right]
        \leq
        C_1
        \delta^{1/C_2}
        \lvert \Lambda_L \rvert^2.
    \end{equation} 
\end{theorem}

\begin{theorem}[Initial scale estimate (ISE)]
    \label{thm:ISE_strong}
Suppose  Assumption~\ref{assumption:epsilon_and_omega} holds, and let $E_0 > 0$ be a lower band edge of $\sigma(\cA_{\omega,L}^\epsilon)$. For all $C_3$, $C_4 > 0$ there exists $\kappa > 0$ such that, if Assumption~\ref{assumption:thinness_density} with exponent $\kappa$ is satisfied, then there exists $L_{\min} \in \epsilon\N$ such that for all $L \in \epsilon\N$ with $L \geq L_{\min}$ we have
    \begin{equation}
    \label{eq:ISE_strong_eq2}
    \PP
    \left[
    \sigma(\cA_{\omega, L}^\epsilon)
        \cap
        [E_0, E_0 + L^{- C_3} ]\neq \emptyset
    \right]
    \leq
    L^{- C_4}.
     \end{equation}
\end{theorem}

\begin{remark}
    Note that, on the one hand, the Wegner estimate does not require any additional constraints on the model parameters $\epsilon$ and $\omega_\pm$ (Assumption~\ref{assumption:epsilon_and_omega}), whereas Theorem~\ref{thm:ISE_strong} does, because to state the latter one needs to know \emph{a priori} that the spectrum of the operator has gaps. On the other hand, Theorem~\ref{thm:Wegner} requires the probability density $\eta_\mu$ to be bounded (Assumption~\ref{assumption:bounded_density}), whereas the ISE requires thinness of the probability measure $\mu$ near the maximum $\omega_+$ of its support (Assumption~\ref{assumption:thinness_density}). All subsequent results, as well as Theorem~\ref{thm:Anderson_localization}, rely on both of the above theorems, and therefore they will require the union of all assumptions to hold.
\end{remark}

The two results above will represent the starting point of a multiscale analysis argument, eventually leading to the proof of localization. In fact, we shall prove the strongest form of localization, namely \emph{strong sub-exponential Hilbert-Schmidt-kernel decay}~\cite[Definition~3.1]{Klein-07}:

\begin{theorem}[Strong sub-exponential Hilbert-Schmidt-kernel decay]
	\label{thm:HS_kernel_decay}
	Suppose that Assumptions~\ref{assumption:bounded_density}, \ref{assumption:thinness_density} and~\ref{assumption:epsilon_and_omega} hold, and let $E_0 > 0$ be a lower band edge of $\sigma(\cA_{\omega,L}^\epsilon)$.
	Then there exists $E_+ > E_0$ such that for all $\zeta \in (0,1)$ we have
	\begin{equation*}
		\label{eq:HS_dynamical_localization}
		\E
		\left[
			\sup_{\lVert f \rVert_\infty \leq 1}
			\left\lVert
				\chi_{\Lambda_1(x)} 
				\chi_{[E_0, E_+]}(\cA_\omega^\epsilon) 
				f(\cA_\omega^\epsilon) 
				\chi_{\Lambda_1(y)}
			\right\rVert_{\mathrm{HS}}^2
		\right]
		\leq
		C_\zeta\,
		\mathrm{e}^{- \lvert x - y \rvert^\zeta}
	\end{equation*}
	for some constant $C_\zeta>0$ and all $x,y \in \epsilon \Z^d$, where the supremum is taken over all Borel-measurable functions. Here $\lVert \,\cdot\, \rVert_{\mathrm{HS}}$ denotes the Hilbert-Schmidt norm.
\end{theorem}

More precisely, Theorems~\ref{thm:ISE_strong} and~\ref{thm:Wegner} imply Theorem~\ref{thm:HS_kernel_decay} via a Boostrap multiscale analysis in the spirit of~\cite{GerminetK-01},
which will be the subject of Section~\ref{sec:MSA}.
It is then well-known that strong sub-exponential Hilbert-Schmidt-kernel decay implies weaker notions of localization --- in particular, \emph{(strong) dynamical localization} (formulated below), as well as spectral localization with exponentially decaying eigenfunctions, commonly referred to as~\emph{Anderson localization} (the claim of our main Theorem~\ref{thm:Anderson_localization}) --- so that Theorems~\ref{thm:Anderson_localization} and~\ref{thm:Dynamical_localization} now follow from Theorem~\ref{thm:HS_kernel_decay} as a corollary, see~\cite[Remark~3.3]{Klein-07}.

\begin{theorem}[Strong dynamical localization]
    \label{thm:Dynamical_localization}
    Under the assumptions of Theorem~\ref{thm:HS_kernel_decay} we have
		\begin{equation*}
		\E
		\left[
			\sup_{\lVert f \rVert_\infty \leq 1}
			\left\lVert
				\lvert X \rvert^{n/2}
				\chi_{[E_0, E_+]}(\cA_\omega^\epsilon) 
				\mathrm{e}^{- i t \cA_\omega^\epsilon}
                \psi
			\right\rVert^2
		\right]
        <
        \infty
		\end{equation*}
        for all $n \in \N$ and all $\psi \in L^2(\R^d)$ with compact support,
        where $\lvert X \rvert$ denotes the operator of multiplication by $x$.
\end{theorem}


\section{Gaps in the spectrum of {$\cA_\omega^\epsilon$}}
\label{sec:Gaps in the spectrum}

Before addressing the proof of our main results, let us show that the problem we are studying is meaningful, in that the spectrum of the operator $\cA_\omega^\epsilon$ has gaps almost surely, provided the parameter $\epsilon$ is sufficiently small. With this knowledge at hand, in the remainder of the paper we will assume to have fixed an $\epsilon$ such that the latter property holds (see Assumption~\ref{assumption:epsilon_and_omega}).

\

Before discussing the technical details, let us outline our strategy in plain English. 
The main issue we need to address is the presence the boundary layer $\cL_{\omega}^\epsilon$, which was required in order to make the coefficients of the operator Lipschitz continuous and apply quantitative unique continuation results. 
The existence of gaps for sufficiently small $\epsilon$ has been proven in recent works on homogenization, cf.~\cite{CherdantsevCV-21} in the absence of a boundary layer.
However, upon observing that the operator $\cA_\omega^\epsilon$ can still be lower and upper bounded by operators with a sharp transition between the two phases,
one can use continuity and monotonicity of the spectra with respect of the size of the inclusions in order to see that also in the presence of a boundary layer, gaps remain.
To exploit this continuity, we shall have to resort to finite-volume approximation of the infinite volume operator.

\

In more technical terms, in this section we prove the following result.

\begin{theorem}[Existence of gaps]
\label{theorem A has gaps}
There exists $\epsilon_0>0$ such that for all $0<\epsilon<\epsilon_0$ the operator $\cA_\omega^\epsilon$ has gapped spectrum. Namely, there are $0<a<b$ such that 
\begin{equation*}
    (a,b)\cap \sigma(\cA_\omega^\epsilon)=\emptyset \qquad \forall \epsilon\in (0,\epsilon_0)
\end{equation*}
almost surely.
\end{theorem}

To this end, let
\begin{equation}
\label{a hat}
\hat a_\omega^\epsilon:=
\begin{cases}
1 & x\in \cM_\omega^\epsilon\cup \cL_\omega^\epsilon\,,
\\
\epsilon^2 &  x\in \cI_\omega^\epsilon\,,
\end{cases}
\end{equation}
and define $\hat{\cA}_\omega^\epsilon$ to be the self-adjoint linear operator in $L^2(\R^d)$ associated with the bilinear form
  \begin{equation}
        \label{eq:definition_A_hat_omega_epsilon_weak}
        \int_{\R^d} \hat a_\omega^\epsilon \nabla u \cdot \nabla v\,, \qquad u,v\in H^1(\R^d)\,.
    \end{equation}
For $L>0$, we denote by $\hat{\cA}_{\omega,L}^\epsilon$ the restriction of $\hat{\cA}_\omega^\epsilon$ to $\Lambda_L$ with Dirichlet boundary conditions.

It was shown in \cite[\S~5.6.4]{CherdantsevCV-21} that $\sigma(\hat{\cA}_\omega^\epsilon)$ has band-gap spectrum almost surely for sufficiently small $\epsilon$ and range of radii $\omega_+-\omega_-$. 
This ensures in particular that Assumption~\ref{assumption:epsilon_and_omega} can be fulfilled for $\hat{\cA}_\omega^\epsilon$.


\

Next, we recall the following well known result, which tells us that the spectra of our operators $\hat \cA_\omega^\epsilon$ and $\cA_\omega^\epsilon$ are contained in the union of the spectra of the corresponding restrictions to boxes of integer side.

\begin{lemma}
\label{lemma spectrum sub spectrum restrictions}
We have
\begin{equation}
\label{lemma spectrum sub spectrum restrictions eq1}
\sigma(\hat{\cA}_\omega^\epsilon) \subseteq \overline{\bigcup_{n\in \N}\sigma(\hat{\cA}_{\omega,n}^\epsilon)}
\end{equation}
and
\begin{equation}
\label{lemma spectrum sub spectrum restrictions eq2}
\sigma(\cA_\omega^\epsilon) \subseteq \overline{\bigcup_{n\in \N}\sigma(\cA_{\omega,n}^\epsilon)}
\end{equation}
almost surely.
\end{lemma}
\begin{proof}
The proof of the first part of~\cite[Proposition 2.1]{Stollmann-98} applies unchanged to our setting.
\end{proof}

\begin{remark}
    \label{rem:actually_equality}
    One actually has identity in~\eqref{lemma spectrum sub spectrum restrictions eq1} and~\eqref{lemma spectrum sub spectrum restrictions eq2}.
    Strictly speaking, equality holds true, in higher generality, for restrictions with \emph{periodic} (as opposed to Dirichlet) boundary conditions, see~\cite[Proposition 2.1]{Stollmann-98}.
    However, since our coefficient matrices are diagonal, one can continue eigenfunctions of the restricted operators beyond $\Lambda_L(x)$ by antisymmetric reflections of the eigenfunctions and symmetric reflections of $a_\omega^\epsilon \mathrm{I}$ and obtain eigenfunctions with periodic boundary conditions on larger boxes.
    Then, \cite[Proposition 2.1]{Stollmann-98} yields an equality in~\eqref{lemma spectrum sub spectrum restrictions eq1} and~\eqref{lemma spectrum sub spectrum restrictions eq2}.
\end{remark}

An immediate, useful consequence of the above lemma is that that one is guaranteed that an interval $I\subset \R$  lies in the resolvent set of the full-space operators as soon as one can ensure that for all $n$ the restrictions of our operators to $\Lambda_n$ have no eigenvalues in $I$.

\begin{corollary}
\label{corollary boxes is enough}
Let $I\subset\R$. Then
\begin{equation*}
\label{corollary boxes is enough eq1}
\sigma(\hat{\cA}_{\omega,n}^\epsilon)\cap I=\emptyset \quad \forall n\in\N\quad \Rightarrow\quad \sigma(\hat{\cA}_\omega^\epsilon) \cap I=\emptyset
\end{equation*}
and
\begin{equation*}
\label{corollary boxes is enough eq2}
\sigma(\cA_{\omega,n}^\epsilon)\cap I=\emptyset \quad \forall n\in\N\quad \Rightarrow\quad \sigma(\cA_\omega^\epsilon) \cap I=\emptyset
\end{equation*}
almost surely.
\end{corollary}

Before supplying the main technical ingredient entering the proof of Theorem~\ref{theorem A has gaps}, let us state and prove a simple lemma.

\begin{lemma}
\label{lemma eigenvalues are continuous}
Let $\delta>0$ be a fixed small parameter.
The map
\begin{equation}
\label{lemma eigenvalues are continuous eq1}
[0,1]\ni t\mapsto \lambda_k(\hat{\cA}_{\omega+t\delta \mathbf{e},n}^\epsilon)
\end{equation}
is continuous for all $k\in \N$.
\end{lemma}
\begin{proof}
Fix a threshold $E_{\mathrm{max}}\in (0,+\infty)$ and suppose $\lambda_k(\hat{\cA}_{\omega,n}^\epsilon)<E_{\mathrm{max}}$. Let $0\le t_1<t_2\le1$. For $j\in\{1,\dots, k\}$, let $\varphi_j^{(2)}$ be the orthonormalised eigenfunctions corresponding to the eigenvalues $\lambda_j(\hat{\cA}_{\omega+t_2\delta \mathbf{e},n}^\epsilon)$. By the Min-Max Theorem we have
\begin{multline}
\label{proof lemma eigenvalues are continuous eq1}
\lambda_k(\hat{\cA}_{\omega+t_1\delta \mathbf{e},n}^\epsilon)=\min_{\dim V=k}\max_{\substack{\varphi\in V,\\ V\subset H^1_0(\Lambda_n),\\\|\varphi\|=1}} \langle \chi_{\Lambda_n}\hat a_{\omega+t_1\delta \mathbf{e}}^\epsilon \nabla \varphi, \nabla\varphi \rangle
\\
\le 
\max_{\substack{\varphi\in\mathrm{span}\{\varphi_1^{(2)}, \dots, \varphi_k^{(2)}\},\\ \|\varphi\|=1}} \langle \chi_{\Lambda_n}\hat a_{\omega+t_2\delta \mathbf{e}}^\epsilon \nabla \varphi, \nabla\varphi \rangle
+
\max_{\substack{\varphi\in\mathrm{span}\{\varphi_1^{(2)}, \dots, \varphi_k^{(2)}\},\\ \|\varphi\|=1}} \langle \chi_{\Lambda_n}(\hat a_{\omega+t_1\delta \mathbf{e}}^\epsilon-\hat a_{\omega+t_2\delta \mathbf{e}}^\epsilon) \nabla \varphi, \nabla\varphi \rangle
\\
=\lambda_k(\hat{\cA}_{\omega+t_2\delta \mathbf{e},n}^\epsilon)
+
\max_{\substack{\varphi\in\mathrm{span}\{\varphi_1^{(2)}, \dots, \varphi_k^{(2)}\},\\ \|\varphi\|=1}} \langle \chi_{\Lambda_n}(\hat a_{\omega+t_1\delta \mathbf{e}}^\epsilon-\hat a_{\omega+t_2\delta \mathbf{e}}^\epsilon) \nabla \varphi, \nabla\varphi \rangle\,.
\end{multline}
Since
\begin{equation*}
\label{proof lemma eigenvalues are continuous eq2}
\epsilon^2 \bigcup_{z\in s_\epsilon+ \epsilon\Z^d}\chi_{B_{\epsilon_{\omega_z+\delta t_2}}(z)\setminus B_{\epsilon_{\omega_z+\delta t_1}}(z)} 
\le
\hat a_{\omega+t_1\delta \mathbf{e}}^\epsilon-\hat a_{\omega+t_2\delta \mathbf{e}}^\epsilon
\le
\bigcup_{z\in s_\epsilon+\epsilon\Z^d}\chi_{B_{\epsilon_{\omega_z+\delta t_2}}(z)\setminus B_{\epsilon_{\omega_z+\delta t_1}}(z)}
\end{equation*}
by \eqref{a hat}, then
\begin{equation}
\label{proof lemma eigenvalues are continuous eq3}
\langle \chi_{\Lambda_n}(\hat a_{\omega+t_1\delta \mathbf{e}}^\epsilon-\hat a_{\omega+t_2\delta \mathbf{e}}^\epsilon) \nabla \varphi, \nabla\varphi \rangle \le C(d,n,E_\mathrm{max}) \,o(|t_2-t_1|)\,.
\end{equation}
Formulae \eqref{proof lemma eigenvalues are continuous eq1} and \eqref{proof lemma eigenvalues are continuous eq3}, in conjunction with the fact that \eqref{lemma eigenvalues are continuous eq1} is monotonically decreasing\footnote{This follows once again from \eqref{a hat} and the Min-Max Theorem.}, yield the claim.
\end{proof}

\begin{theorem}
\label{theorem gaps are protected}
Let $I:=(I_-,I_+)\subset \R$ be an interval. Fix $\epsilon\in(0,1)$ and let $\delta:=\epsilon^\gamma/4$. Suppose that
\begin{equation}
\label{theorem gaps are protected eq1}
\sigma(\hat{\cA}_{\omega+t\delta \mathbf{e}}^\epsilon)\cap I=\emptyset \qquad \forall t\in[0,1]
\end{equation}
almost surely.
Then we have
\begin{equation*}
\label{theorem gaps are protected eq2}
\sigma(\cA_\omega^\epsilon) \cap I =\emptyset
\end{equation*}
almost surely.
\end{theorem}

\begin{remark}
From a geometric perspective, the above theorem tells us the following: if we know that, by simultaneously and continuously increasing the radii of \emph{all} sharp (i.e., without boundary layer) inclusions by an amount between $0$ and $\delta$, the spectrum of the resulting (sharp) operator never falls in the interval $I$ --- the content of formula~\eqref{theorem gaps are protected eq1}, then 
no spectrum of the operator $\cA_\omega^\epsilon$ can fall in $I$ either. Note that $\delta$ is precisely the size of the boundary layer of `smoothened' inclusions, so that the perturbation of radii of sharp inclusions ranges between the inner and the outer boundary of the boundary layer of `smoothened' inclusions.
We refer to Figure~\ref{fig:a_and_a_hat} for an illustration.
\end{remark}

\begin{figure}[ht]
\begin{center}
\begin{tikzpicture}
\begin{scope}
    \draw[thick, ->] (-4,0) -- (4.5,0);
    \draw[thick, ->] (0,-.1) -- (0,2.25);

    \begin{scope}[yshift = -.0cm]
    \draw[thick] (-4,2) -- (-3,2) -- (-2,.1) -- (2,.1) -- (3,2) -- (4,2);
    \end{scope}

    \begin{scope}[yshift = -0.0cm]
    \draw[thick, red]           (-4,2) -- (-3,2);
    \draw[thick, dashed, red]   (-3,2) -- (-3,.1);
    \draw[thick, red]           (-3,.1) -- (3,.1);
    \draw[thick, dashed, red]   (3,.1) -- (3,2);
    \draw[thick, red]           (3,2) -- (4,2);
    \end{scope}

    \begin{scope}[yshift = .0cm]
    \draw[thick, blue]           (-4,2) -- (-2.5,2);
    \draw[thick, dashed, blue]   (-2.5,2) -- (-2.5,.1);
    \draw[thick, blue]           (-2.5,.1) -- (2.5,.1);
    \draw[thick, dashed, blue]   (2.5,.1) -- (2.5,2);
    \draw[thick, blue]           (2.5,2) -- (4,2);
    \end{scope}

    \begin{scope}[yshift = .0cm]
    \draw[thick, teal]           (-4,2) -- (-2,2);
    \draw[thick, dashed, teal]   (-2,2) -- (-2,.1);
    \draw[thick, teal]           (-2,.1) -- (2,.1);
    \draw[thick, dashed, teal]   (2,.1) -- (2,2);
    \draw[thick, teal]           (2,2) -- (4,2);
    \end{scope}

    \draw[very thick] (-.1,.1) -- (.1,.1);
    \draw (-.25,.4) node {$\epsilon^2$};
    \draw[very thick] (-.1,2) -- (.1,2);
    \draw (-.25,2) node {$1$};
\end{scope}

\begin{scope}[xshift = 5cm]

    \draw[thick] (0,1.75) -- (.75,1.75);
    \draw[thick, teal] (0,1.25) -- (.75,1.25);
    \draw[thick, blue] (0,.75) -- (.75,.75);
    \draw[thick, red] (0,.25) -- (.75,.25);

    \draw[anchor = west] (1,1.75) node {$a_\omega^\epsilon$};
    \draw[anchor = west] (1,1.25) node {$\hat{a}_\omega^\epsilon$};
    \draw[anchor = west] (1,.75) node {$\hat{a}_{\omega + t \delta \mathbf{e}}^\epsilon$};
    \draw[anchor = west] (1,.25) node {$\hat{a}_{\omega + \delta \mathbf{e}}^\epsilon$};
    
\end{scope}

\end{tikzpicture}
\end{center}

\caption{Geometric visualisation of Theorem~\ref{theorem gaps are protected}.}
\label{fig:a_and_a_hat}
\end{figure}

\begin{proof}[Proof of Theorem~\ref{theorem gaps are protected}]
In view of Corollary~\ref{corollary boxes is enough}, it suffices to show that 
\begin{equation}
\label{proof theorem gaps are protected eq1}
\sigma(\cA_{\omega,n}^\epsilon) \cap I =\emptyset \qquad \forall n \in \N
\end{equation}
almost surely.

To this end, let us begin by observing that the explicit formulae~\eqref{a_epsilon} and~\eqref{a hat} for our coefficients imply
$
\hat{\cA}_{\omega+\delta \mathbf{e},n}^\epsilon\le \cA_{\omega,n}^\epsilon\le\hat{\cA}_{\omega,n}^\epsilon\,,
$
which, in turn, gives us
\begin{equation}
\label{proof theorem gaps are protected eq7}
\lambda_k(\hat{\cA}_{\omega+\delta \mathbf{e},n}^\epsilon)
\le 
\lambda_k(\cA_{\omega,n}^\epsilon)
\le
\lambda_k(\hat{\cA}_{\omega,n}^\epsilon)\,.
\end{equation}
Arguing by contradiction, suppose there exists $k\in \N$ such that 
\begin{equation}
\label{proof theorem gaps are protected eq8}
\lambda_k(\cA_{\omega,n}^\epsilon)\in I\,.
\end{equation}
Then formulae \eqref{theorem gaps are protected eq1}, \eqref{proof theorem gaps are protected eq7}, and \eqref{proof theorem gaps are protected eq8} imply
\begin{equation*}
\label{proof theorem gaps are protected eq9}
\lambda_k(\hat{\cA}_{\omega+\delta \mathbf{e},n}^\epsilon)\le I_- 
\qquad 
\text{and} 
\qquad 
\lambda_k(\hat{\cA}_{\omega,n}^\epsilon)\ge I_+.
\end{equation*}
But Lemma~\ref{lemma eigenvalues are continuous} would then imply $\lambda_k(\hat{\cA}_{\omega+t\delta \mathbf{e},n}^\epsilon)\in I=(I_-,I_+)$ for some $t\in[0,1]$, thus contradicting \eqref{theorem gaps are protected eq1}.
Hence \eqref{proof theorem gaps are protected eq8} cannot hold and we must have \eqref{proof theorem gaps are protected eq1}.
\end{proof}

Equipped with the above results, we are finally in a position to prove Theorem~\ref{theorem A has gaps}.

\begin{proof}[Proof of Theorem~\ref{theorem A has gaps}]
It follows from~\cite[\S~5.6.4]{CherdantsevCV-21} that, for sufficiently small $\epsilon$, the spectra of both $\hat{\cA}_{\omega}^\epsilon$ and $\hat{\cA}_{\omega+\delta \mathbf{e}}^\epsilon$, $\delta:=\epsilon^\gamma/4$, have gaps. By Lemma~\ref{lemma eigenvalues are continuous} we then have \eqref{theorem gaps are protected eq1}, so that Theorem~\ref{theorem gaps are protected} yields the claim.
\end{proof}

\section{Wegner estimate}
\label{Wegner estimate}

This section is devoted to the proof of the Weger estimate~\eqref{eq:Wegner}.

\

Let us begin by stating and proving a helpful lemma on eigenvalue lifting.

\begin{lemma}
\label{lem:eigenvalue_lifting_reformulated}
    Let $\epsilon > 0$ and $0 < \tilde E_- < \tilde E_+ < \infty$.
    There exists constants $\tau > 1$ and $s_0 > 0$ such that 
    we have
    \begin{equation}
\label{eq:lem_eigenvalue_lifting_eq2}
    \lambda_k(\cA_{\omega- s \mathbf{e}, L}^\epsilon)
    \geq
    \lambda_k(\cA_{\omega, L}^\epsilon)
    +
    s^\tau \qquad \forall s\in[0,s_0]\,,
    \end{equation}
    for all $k \in \N$, $L \in \epsilon \N$, and $\omega \in [\omega_-,\omega_+]^{\Z^d}$ for which
    \begin{equation*}
\label{eq:lem_eigenvalue_lifting_eq1}
        \tilde E_- 
    \leq \lambda_k(\cA_{\omega, L}^\epsilon)
    \leq
    \lambda_k(\cA_{\omega - s_0 \mathbf{e}, L}^\epsilon)
    \leq
    \tilde E_+\,.
    \end{equation*}
\end{lemma}

\begin{proof}
For $s=0$ the claim \eqref{eq:lem_eigenvalue_lifting_eq2} is trivially true; therefore, in the remainder of the proof we will assume $s>0$.

Let $X=\{x_z\}_{z\in \epsilon\Z^d}$ be the equidistributed family of points from \eqref{eq:coeff_property_2a} and put
\begin{equation}
\label{eq:proof_lem_eigenvalue_lifting_eq0}
    S_{\rho,X}(L):=\bigcup_{z \in \epsilon \Z^d \, :\, z+s_\epsilon\in\Lambda_L }B_{\rho}(x_z)\,.
\end{equation}
    Formula \eqref{eq:coeff_property_2a} (with $\omega$ replaced by $\omega - s \mathbf{e}$ in the RHS) together with \eqref{eq:proof_lem_eigenvalue_lifting_eq0} and the Min-max Lemma imply that
%
%
\begin{equation*}
\label{eq:proof_lem_eigenvalue_lifting_eq1}
    \lambda_k(\cA_{\omega-s\mathbf{e}, L}^\epsilon) 
    \ge 
    \lambda_k
    \left(
    \left.
    - \nabla
    (
    a_{\omega}^\epsilon+\alpha_\epsilon \,s \,\chi_{S_{\frac{\epsilon^\gamma}{10}s,X}(L)}
    )
    \nabla
    \right|_{\Lambda_L}
    \right).
\end{equation*}

Hence, in order to prove \eqref{eq:lem_eigenvalue_lifting_eq2} it suffices to show that, for sufficiently small $s$, one has
\begin{equation}
\label{eq:proof_lem_eigenvalue_lifting_eq2}
    \lambda_k
    \left(
    \left.
        -\nabla
        (
        a_{\omega}^\epsilon+\alpha_\epsilon \,s \,\chi_{S_{\frac{\epsilon^\gamma}{10}s,X}(L)}
        )
        \nabla
        \right|_{\Lambda_L}
    \right)
   \ge
    \lambda_k(\cA_{\omega, L}^\epsilon)+s^\tau. 
\end{equation}
Although~\eqref{eq:proof_lem_eigenvalue_lifting_eq2} follows, in principle, rather easily from \cite[Lemma~4.2]{Dicke-21}, we provide a full proof in our setting for the reader's convenience, as the latter demonstrates the role played by our assumptions on the regularity of the diffusion coefficients.

The key ingredient is the quantitative unique continuation principle for gradients of eigenfunctions proved in \cite{DickeV-23}, in the form of \cite[Corollary~6.5]{DickeV-23}, which one would like to apply to the operator $\left.
        -\nabla
        (
        (
        a_{\omega}^\epsilon+\alpha_\epsilon \,s \,\chi_{S_{\frac{\epsilon^\gamma}{10}s,X}(L)}
        )
        \nabla
        )
        \right|_{\Lambda_L}$. Now, \cite[Corollary~6.5]{DickeV-23} requires the coefficients the operator to be Lipschitz. Our assumptions on the diffusion coefficients ensure that the contribution from $a_\omega^\epsilon$ is such. Furthermore, it is not hard to see that one can find a Lipschitz continuous function $g$ satisfying
        \begin{equation}
        \label{eq:proof_lem_eigenvalue_lifting_eq2bis}
        \alpha_\epsilon \,s \,\chi_{S_{\frac{\epsilon^\gamma}{20}s,X}(L)}
        \le
        g
        \le
        \alpha_\epsilon \,s \,\chi_{S_{\frac{\epsilon^\gamma}{10}s,X}(L)}
        \end{equation}
        with Lipschitz constant of the order of $\alpha_\epsilon \epsilon^{-\gamma}$. 
        Then, upon resorting to the Min-max Lemma once more with account of \eqref{eq:proof_lem_eigenvalue_lifting_eq2bis}, \cite[Corollary~6.5]{DickeV-23} applied to the operator $\left.
        -\nabla
        (
        (a_{\omega}^\epsilon+g)
        \nabla
        )
        \right|_{\Lambda_L}$ yields
\begin{multline}
\label{eq:proof_lem_eigenvalue_lifting_eq3}
    \lambda_k
    \left(
    - \nabla
    (
    a_{\omega}^\epsilon+t\alpha_\epsilon \,s \,\chi_{S_{\frac{\epsilon^\gamma}{10}s,X}(L)}
    \nabla
    )
    \mid_{\Lambda_L}
    \right)
\ge
\lambda_k(\cA_{\omega, L}^\epsilon)
+
t
\tilde E_-^2 \alpha_\epsilon \,s\left(\frac{\epsilon^{\gamma-1}}{40}s \right)^{N(1+\tilde E_+^{2/3})}
\\
=
\lambda_k(\cA_{\omega, L}^\epsilon)
+
t\,\mathfrak{c}_\epsilon\,
s^{N(1+\tilde E_+^{2/3})+1}
\end{multline}
for all  $t\in[0,1]$ and some constants $N>0$ depending on $\epsilon$, $\tilde E_+$ and $\tilde E_-$. In \eqref{eq:proof_lem_eigenvalue_lifting_eq3} we defined $\mathfrak{c}_\epsilon:=\tilde E_-^2 \alpha_\epsilon \left(\frac{\epsilon^{\gamma-1}}{40} \right)^{N(1+\tilde E_+^{2/3})}$. Recalling~\eqref{eq:alpha_epsilon} and the fact that $\epsilon\in(0,1)$, 
by choosing $N$ sufficiently large, one can make $\mathfrak{c}_\epsilon$ arbitrarily small, and certainly smaller than $1$. Therefore, for $0< s\le s_0<1$, one has
\[
(s^\frac{1}{|\ln s|})^{|\ln \mathfrak{c}_\epsilon|} \ge (s^\frac{1}{|\ln s_0|})^{|\ln \mathfrak{c}_\epsilon|}
\]
so that
\begin{equation}
\label{eq:proof_lem_eigenvalue_lifting_eq4}
    \mathfrak{c}_\epsilon= s^{\frac{|\ln\mathfrak{c}_\epsilon|}{|\ln s|}}\ge s^{C(\epsilon,d,\tilde E_\pm)}\,.
\end{equation}
By combining \eqref{eq:proof_lem_eigenvalue_lifting_eq3} and \eqref{eq:proof_lem_eigenvalue_lifting_eq4} one obtains \eqref{eq:proof_lem_eigenvalue_lifting_eq2} for an appropriate choice of a sufficiently large $\tau$, thus proving the claim.
\end{proof}

Equipped with the above result, we can move on to the proof of the Wegner estimate.

\begin{proof}[Proof of Theorem~\ref{thm:Wegner}]
Let us fix $0<E_-<E_+<+\infty$. 

Let $\rho_\delta\in C^\infty(\R^d,[-1,0])$, $\delta>0$, be a non-decreasing function such that:
\begin{enumerate}[(i)]
    \item $\rho_\epsilon=-1$ for $x\le \delta$, $\rho=0$ for $x\ge \delta$;
    \item $\|\rho'\|_\infty\le \delta^{-1}$;
    \item we have 
    \begin{equation}
    \label{eq:cutoff rho support}
        \chi_{[E-\delta,E+\delta]}(x)\le \rho(x-(E-2\rho))-\rho(x-(E+2\rho)) \le \chi_{[E-3\delta,E+3\delta]}(x)
    \end{equation} 
    for all $x\in \R^d$.
\end{enumerate}
Here we assume $\delta$ to be sufficiently small (a more explicit constraint will be imposed later on) and fixed, and $E\in (E_-,E_+)$ to be such that $\min_{\aleph\in\{+,-\}}\mathrm{dist}(E,E_\aleph)<3\delta$.
Note that middle quantity in the inequality \eqref{eq:cutoff rho support} has compact support as a function of $x$.

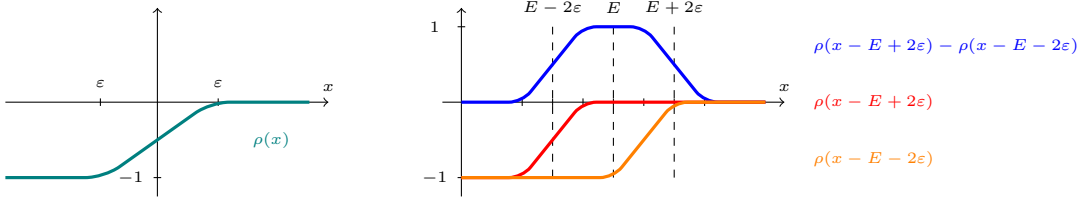
\begin{center}
\begin{tikzpicture}
    \begin{scope}[xshift = 0cm]

    \draw[->] (-2,0) -- (2.25,0);
    \draw[->] (0,-1.25) -- (0,1.25);

    \draw (-.05,-1) -- (0.05,-1);
    \draw (-0.35,-1) node {\tiny $-1$};

    \draw (-.75,-.05) -- (-.75,.05);
    \draw (-.75,.25) node {\tiny $\epsilon$};

    \draw (.8,-.05) -- (.8,.05);
    \draw (.8,.25) node {\tiny $\epsilon$};

    \draw (2.25,.2) node {\tiny $x$};

    \draw[very thick, rounded corners = 7pt, teal]
    (-2,-1) -- (-.7,-1) -- (0,-.5) -- (.7,0) -- (2,0);

    \draw (1.5,-.5) node {\tiny$\color{teal}\rho(x)$};
    \end{scope}

    \begin{scope}[xshift = 4cm]

    \draw[->] (-.25,0) -- (4.25,0);
    \draw[->] (0,-1.25) -- (0,1.25);

    \draw (-.05,-1) -- (0.05,-1);
    \draw (-0.35,-1) node {\tiny $-1$};

    \draw (-.05,1) -- (0.05,1);
    \draw (-0.35,1) node {\tiny $1$};    

    \draw (.8,-.05) -- (.8,.05);
    
    \draw (1.2,-.05) -- (1.2,.05);
    \draw[thin, dashed] (1.2,-1) -- (1.2,1);    
    \draw (1.2,1.25) node {\tiny $E - 2 \epsilon$};

    \draw (1.6,-.05) -- (1.6,.05);
    
    \draw (2,-.05) -- (2,.05);
    \draw[thin, dashed] (2,-1) -- (2,1);    
    \draw (2,1.25) node {\tiny $E$};

    \draw (2.4,-.05) -- (2.4,.05);

    \draw (2.8,-.05) -- (2.8,.05);
    \draw[thin, dashed] (2.8,-1) -- (2.8,1);
    \draw (2.8,1.25) node {\tiny $E + 2 \epsilon$};    

    \draw (3.2,-.05) -- (3.2,.05);

    \draw (4.25,.2) node {\tiny $x$};

    \draw[very thick, rounded corners = 5pt, blue]
    (-0,0) -- (.8,0) -- (1.6,1) -- (2.4,1) -- (3.2,0) -- (4,0);

    \draw[very thick, rounded corners = 5pt, red]
    (-0,-1) -- (.8,-1) -- (1.6,0) -- (4,0);

    \draw[very thick, rounded corners = 5pt, orange]
    (-0,-1) -- (2.0,-1) -- (2.8,0) -- (4,0);    

    \draw[anchor = west] (4.5,.75) node {\tiny \color{blue}$\rho(x-E+2\epsilon) - \rho(x-E-2\epsilon)$};
    
    \draw[anchor = west] (4.5,0) node {\tiny \color{red}$\rho(x-E+2\epsilon)$};

    \draw[anchor = west] (4.5,-.75) node {\tiny \color{orange}$\rho(x-E-2\epsilon)$};

    \end{scope}

\end{tikzpicture}
 
\end{center}

The inequality~\eqref{eq:cutoff rho support} implies
\begin{multline}
\label{eq:proof Wegner eq 1}
     \E
        \left[
            \operatorname{Tr}
            \left[
               \chi_{[E-\delta, E+\delta]}
                (
                \cA_{\omega, L}^\epsilon
                )
            \right]
        \right]
        \\
        \le
        \E\left[
        \sum_{k\in \N} \left[\rho(\lambda_k(\cA_{\omega, L}^\epsilon)-(E-2\rho))-\rho(\lambda_k(\cA_{\omega, L}^\epsilon)-(E+2\rho))\right]
        \right].
\end{multline}
Hence, the task at hand reduces to estimating the sum in the RHS of~\eqref{eq:proof Wegner eq 1}. 

We observe that the latter sum is, in fact, finite. Indeed, because of~\eqref{eq:cutoff rho support}, only eigenvalues $\lambda_k(\cA_{\omega, L}^\epsilon)$ such that
\begin{equation}
    \label{eq:proof Wegner eq 2}
    \lambda_k(\cA_{\omega, L}^\epsilon) \in [E-3\delta,E+3\delta]\subset[E_-,E_+]
\end{equation}
give a nonzero contribution to the RHS of~\eqref{eq:cutoff rho support}.

We would like to find an upper bound on $\lambda_k(\cA_{\omega-s_0 \mathbf{e}, L}^\epsilon)$ for sufficiently small $s_0 > 0$. In other words, we would like to establish how much individual eigenvalues can lift if one simultaneously decreases the radii of all inclusions by $s_0$.
For this purpose, note that by assumption we have
\begin{equation*}
\label{eq:proof Wegner eq 3}
    \lambda_k(- \epsilon^2 \Delta \mid_{\Lambda_L})
    \leq
\lambda_k(\cA_{\omega-s_0 \mathbf{e}, L}^\epsilon)
    \leq 
    E_+\,,
\end{equation*}
whence, by (inverse) Weyl's asymptotics, for all $k$ contributing to the sum in \eqref{eq:proof Wegner eq 2} we obtain 
\begin{equation}
\label{eq:proof Wegner eq 4}
    k \leq K_1 \left( \frac{E_+}{\epsilon} \right)^{\frac{d}2} L^d\,,
\end{equation}
$K_1$ being a universal constant.
At the same time, denoting by $\Delta_L$ the (negative) Laplacian on $\Lambda_L$ with Dirichlet boundary conditions and resorting to Weyl's asymptotics once again, we get
\begin{equation}
\label{eq:proof Wegner eq 5}
    \lambda_k(\cA_{(\omega_--s_0) \mathbf{e}, L}^\epsilon)
    \overset{\eqref{eq:coeff property 1}}{\leq}
    \lambda_k( - \Delta_L)
    \leq
    K_2 \frac{k^{\frac2d}}{L^2}
    \overset{\eqref{eq:proof Wegner eq 4}}{\leq}
    \frac{K_1^{2/d} K_2 }{\epsilon} E_+\,,
\end{equation}
where $K_2$ is another universal constant.
All in all, we conclude
\begin{equation*}
\label{eq:proof Wegner eq 6}
     \tilde E_-
    :=
    E_-
    \leq
    \lambda_k(\cA_{\omega, L}^\epsilon)
    \leq
    \lambda_k(\cA_{\omega-s_0 \mathbf{e}, L}^\epsilon)
    \leq
    \frac{K_1^{2/d} K_2 }{\epsilon} E_+
    =:
    \tilde E_+.
\end{equation*}
This allows one to apply Lemma~\ref{lem:eigenvalue_lifting_reformulated} with $s = (4 \rho)^{1/\tau}$, where $\tau$ is the parameter from Lemma~\ref{lem:eigenvalue_lifting_reformulated} (depending on $E_+, E_-$ and thus also on $\epsilon$), and we conclude
\begin{equation}
\label{eq:proof Wegner eq 7}
    \lambda_k(\cA_{\omega- s \mathbf{e}, L}^\epsilon)
    \geq
    \lambda_k(\cA_{\omega, L}^\epsilon)+4\rho.
\end{equation}
for all $\rho\le \rho_0$ and $0<s\le s_0$, where $\rho_0:=\tfrac{s_0^\tau}4$, $s_0$ being the second parameter from Lemma~\ref{lem:eigenvalue_lifting_reformulated}. Note that $0<\rho\le\rho_0<s_0<1$, as one can easily establish by direct inspection. 

Formulae~\eqref{eq:cutoff rho support} and~\eqref{eq:proof Wegner eq 7} then imply
\begin{multline}
    \label{eq:proof_Wegner_eq_3}
    \E\left[
        \sum_{k\in \N} \left[\rho(\lambda_k(\cA_{\omega, L}^\epsilon)-(E-2\rho))-\rho(\lambda_k(\cA_{\omega, L}^\epsilon)-(E+2\rho))\right]
        \right]
        \\
        \le
        \E\left[
        \sum_{k\in \N} \left[ \rho(\lambda_k(\cA_{\omega- s \mathbf{e}, L}^\epsilon)-(E+2\rho))-\rho(\lambda_k(\cA_{\omega, L}^\epsilon)-(E+2\rho)) \right]
        \right].
\end{multline}

The rest of the proof consists in estimating the RHS of~\ref{eq:proof_Wegner_eq_3} from above, and retraces arguments from~\cite[Theorem~3.1]{Dicke-21}, see also~\cite{HundertmarkKNSV-05, NakicTTV-18}:
\begin{itemize}
    \item First, one notes that, by a Weyl-type bound, since $\operatorname{supp} \rho \subset ( - \infty, \delta] \subset ( - \infty, \delta_{\max}]$, only $\sim L^d$ many $k \in \N$ can contribute to the sum in the RHS of~\eqref{eq:proof_Wegner_eq_3}.
    \\
    \textbf{This yields a factor $L^d$ in the upper bound.}
    \item
    The differences in the RHS of~\eqref{eq:proof_Wegner_eq_3} are between a function of an eigenvalue of the operator $\cA_{\omega, L}^\epsilon$, that is the standard configuration, and a function of an eigenvalue of $\cA_{\omega- s \mathbf{e}, L}^\epsilon$, that is the configuration where all $\sim L^d$ many inclusion in $\Lambda_L$ are shrunk. 
    By changing the radius of one inclusion at a time, the sum can be recast as a telescopic sum of $\sim L^d$ many differences of functions of eigenvalues of operators differing in only one site.
    \\
    \textbf{This yields another factor $L^d$ in the upper bound.}
    \item
    In each of these differences, we first evaluate the expectation with respect to the one site where the change takes place. 
    This is a one-dimensional integral over an interval of length $s = \rho^{1/\kappa}$.
    Using an integration-by-parts argument as in~\cite[Proof of Lemma~4.5]{Dicke-21} and the boundedness of the probability density $\eta_\mu$, the latter integral can be estimated from above by a term proportional to $\rho^{1/\kappa}$.
    \\
    \textbf{This will contribute a factor $\rho^{1/\kappa}$ in the upper bound}.
    \item 
    The expectation over the remaining random variables are estimated uniformly, leveraging on the boundedness of $\rho$, which concludes the proof.
    \qedhere
    \end{itemize}
\end{proof}

\section{Proof of initial scale estimate}
\label{Initial scale estimate under the strong assumption}

In this section, we prove our initial scale estimate, that is Theorem~\ref{thm:ISE_strong}.
We emphasise that we rely on Assumption~\ref{assumption:thinness_density} on thinness of the density $\nu_\mu$ of the random variables near their maximum.
Indeed, if one had more detailed information on the behaviour of the quantitative unique continuation principle in~\cite{DickeV-23} under scaling (more precisely, if the dependence of the constant $N$ in~\cite[Corollary 6.2]{DickeV-23} on $G \vartheta_{\mathrm{Lip}}$ were explicit and polynomial) this would allow one to resort to more modern techniques as used in~\cite{SeelmannT-20} and drop Assumption~\ref{assumption:thinness_density}.

\begin{proof}[Proof of Theorem~\ref{thm:ISE_strong}]
    We follow a strategy similar to that in~\cite{BarbarouxCH-97, KirschSS-98}.
    
    For $0<s<\omega_+-\omega_-$ and $L \in \epsilon \N $, consider the event
    \begin{equation}
    \label{eq:proof_ISE_strong_eq1}
        A_{s,L}
        :=
        \left\{
            \omega \in \Omega
            \ \colon\
            \omega_z < \omega_+ - s
            \ \text{for all}
            \
            z \in \Z^d 
            \
            \text{with}
            \
            \epsilon z+s_\epsilon \in \Lambda_L
        \right\}.
    \end{equation}
    Using Assumption~\ref{assumption:thinness_density}, one can estimate the probability of the event~\eqref{eq:proof_ISE_strong_eq1} as
    \begin{multline}
    \label{eq:proof_ISE_strong_eq2}
    \PP
    \left[
        A_{s,L}
    \right]
    =
    1 - \PP
    \left[
        \bigcup_{\epsilon z+s_\epsilon \in \Lambda_L}
        \left\{\omega \in \Omega \ \colon \
            \omega_z \geq \omega_+ - s
        \right\}
    \right]
    \\
    \geq
    1 
    -
    \# 
    \left\{
        z \in \Z^d \colon \epsilon z +s_\epsilon \in \Lambda_L
    \right\}
    \cdot
    \PP \left[\{\omega\in \Omega\ \colon \ \omega_0 \geq \omega_+ - s\}\right]
    \\
    \geq
    1
    -
    \left(
        \frac{L}{\epsilon}
    \right)^d
    s^\kappa\,,
    \end{multline}
    where the constant $\kappa>0$ will be chosen later.

    In light of Remark~\ref{rem:actually_equality} the restrictions $\cA_{\omega, L}^\epsilon$ have no spectrum in gaps of $\cA_\omega^\epsilon$, almost surely.    
    In particular, since $E_0$ is a lower spectral band edge of $\cA_{\omega}^\epsilon$, there exists a unique $k \in \N$ (depending on $L$) such that $\lambda_k(\cA_{\omega, L}^\epsilon) = \inf \{ \sigma(\cA_{\omega, L}^\epsilon) \cap [E_0, \infty)\}$ for all $\omega \in \Omega$.
 
    Furthermore, there exist a positive real number $E_{\mathrm{max}} > E_0$ and an $\tilde L_{\mathrm{max}}>0$ such that 
    \begin{equation}
       \label{eq:proof_ISE_strong_eq3}
\lambda_k(\cA_{\omega, L}^\epsilon) \in [E_0, E_\mathrm{max}] \qquad \text{for all $L>\tilde L_{\mathrm{min}}$ and all $\omega \in \Omega$,}
\end{equation}
        as one can establish by a similar Weyl type argument as in~\eqref{eq:proof Wegner eq 5}.
    In particular, for $L>\tilde L_{\mathrm{min}}$ we have that
\begin{equation} 
\label{eq:proof_ISE_strong_eq4}
    \lambda_k(\cA_{\omega,L}^\epsilon)
    \geq
    \lambda_k(\cA_{\omega_+ \, \mathbf{e},L}^\epsilon)
    +
    L^{-C_3} 
    \quad
    \Rightarrow
    \quad 
    \sigma(\cA_{\omega,L}^\epsilon) \cap [E_0, E_0 + L^{-C_3}] = \emptyset\,,
\end{equation}
because $\lambda_k(\cA_{\omega_+ \mathbf{e},L}^\epsilon)\in [E_0,E_\mathrm{max}]$ by~\eqref{eq:proof_ISE_strong_eq3} and $\lambda_k(\cA_{\omega,L}^\epsilon)$ is the smallest eigenvalue of $\cA^\epsilon_{\omega,L}$ in $[E_0,+\infty)$.

Thanks to Lemma~\ref{lem:eigenvalue_lifting_reformulated}, there exist $\tau > 0$ and $0<s_0<\omega_+-\omega_-$ such that
    \[
    \lambda_k(\cA_{(\omega_+ - s) \, \mathbf{e},L}^\epsilon)
    \geq
    \lambda_k(\cA_{\omega_+ \, \mathbf{e},L}^\epsilon)
    +
    s^\tau \qquad \forall s \in[0,s_0]\,.
    \]
    Thus, for $\omega \in A_{s,L}$ and $s\in[0,s_0]$ we have
    \begin{equation} 
    \label{eq:proof_ISE_strong_eq5}
    \lambda_k(\cA_{\omega,L}^\epsilon)
    \geq
    \lambda_k(\cA_{(\omega_+ - s) \, \mathbf{e},L}^\epsilon)
    \geq
    \lambda_k(\cA_{\omega_+ \, \mathbf{e},L}^\epsilon)
    +
    s^\tau\,.
    \end{equation} 
    Formulae~\eqref{eq:proof_ISE_strong_eq2}, \eqref{eq:proof_ISE_strong_eq4} and~\eqref{eq:proof_ISE_strong_eq5} then imply
    \begin{equation} 
    \label{eq:proof_ISE_strong_eq6}
    \PP
    \left[
        \sigma(\cA_{\omega,L}^\epsilon)
        \cap
        [E_0, E_0 + s^\tau]
        =
        \emptyset
    \right]
    \geq
    \PP [A_{s, L}]
    \geq
    1
    -
    \left(
        \frac{L}{\epsilon}
    \right)^d
    s^\kappa.
    \end{equation}
    By choosing $L_{\mathrm{min}}\ge\tilde L_{\mathrm{min}}$ sufficiently large depending on $\epsilon$, $s_0$, $C_3$ and $\tau$, one can ensure that
    $\epsilon^{-d}\le L$ and that there is a unique $s_L\in[0,s_0]$ such that $s_L= L^{-C_3/\tau}$ for all $L\ge L_{\mathrm{min}}$. Formula~\eqref{eq:proof_ISE_strong_eq6} with $s=s_L$ then implies
    \[
    \PP
    \left[
        \sigma(\cA_{\omega,L}^\epsilon)
        \cap
        [E_0, E_0 + L^{-C_3}]
        \neq
        \emptyset
    \right]    
        \leq
    \epsilon^{-d}
    L^{d +1- \frac{\kappa C_3}{\tau}} \qquad \forall L\ge L_{\mathrm{min}}.
    \]
    By choosing $\kappa$ such that
    \[
    d +1- \frac{\kappa C_3}{\tau}
    <
    - C_4
    \quad
\Leftrightarrow
    \quad
    \kappa 
    >
    \frac{\tau (d+1 + C_4)}{C_3}
    \]
    we arrive at~\eqref{eq:ISE_strong_eq2}. 
\end{proof}

\section{Multiscale analysis}
\label{sec:MSA}

The purpose of this section is to provide the proof of Theorem~\ref{thm:HS_kernel_decay}, that is to show that the initial scale estimate, Theorem~\ref{thm:ISE_strong}, and the Wegner estimate, Theorem~\ref{thm:Wegner}, imply strong sub-exponential Hilbert-Schmidt kernel decay in the sense of Theorem~\ref{thm:HS_kernel_decay} and thus dynamical and Anderson localization in the sense of Theorems~\ref{thm:Dynamical_localization} and~\ref{thm:Anderson_localization}.

For this, we shall resort to the Bootstrap Multiscale Analysis.
Before citing~\cite[Theorem and 3.8]{GerminetK-01}, we have to introduce a number of assumptions on the random family of operators $(\cA_{\omega}^\epsilon)_{\omega \in \Omega}$ and their restrictions $\cA_{\omega,L,x}^\epsilon$ to boxes $\Lambda_L(x)$ with Dirichlet boundary conditions, and verify that these are fulfilled.

Recalling that $\Lambda_L(x)$ denotes the box $x + (0,L)^d$, and $s_L = (\frac{L}{2}, \dots, \frac{L}{2}) \in \R^d$, we shall adopt the following definitions and notation.
\begin{definition}[Boundary of boxes]
    \label{def:Gamma}
    Given $x \in \epsilon \Z^d$ and $L\in \epsilon \N$, we define
    $\Gamma_{L}(x)$ as the inner boundary belt of thickness $\epsilon$ and at distance $\epsilon/2$ to the boundary of $\Lambda_L(x)$, that is
        \[
    \Gamma_L(x)
    :=
    \overline{\Lambda_{L - \epsilon}(x + s_{\epsilon})} 
    \backslash  
    \Lambda_{L-3 \epsilon} \left(x + s_{3 \epsilon}
    \right),
    \]
    see Figure~\ref{fig:Lambda_and_Gamma} for an illustration.
\end{definition}

\begin{figure}
    \centering
\begin{tikzpicture}
  \begin{scope}   
    
    \fill[even odd rule, pattern = north east lines] 
    (.5,.5) rectangle (4.5,4.5)
    (1.5,1.5) rectangle (3.5,3.5);
    \draw[thick, dashed] (0,0) rectangle (5,5);

    \draw[thick, fill = white] (0,0) circle (2pt);  
    \draw (0,-.25) node {$x$};
    \draw[thick, fill = white] (.5,.5) circle (2pt);
    \draw (1,.25) node {$x + s_\epsilon$};
    
    \draw[thick, fill = white] (1.5,1.5) circle (2pt);
    \draw[anchor = west] (1.5,1.75) node {$x + s_{3 \epsilon}$};     
    \end{scope}

    \begin{scope}[xshift = 6cm]
    \fill[pattern = north east lines]  (0,.5) rectangle (2,1.5);
    \draw[thick, dashed] (0,2.5) rectangle (2,3.5);

    \draw[anchor = west] (2,3) node {$\Lambda_L(x)$};
    \draw[anchor = west]  (2,1) node {$\Gamma_L(x)$};

    \end{scope}
\end{tikzpicture}

\caption{Illustration of the sets $\Lambda_L(x)$ and $\Gamma_L(x)$ in Definition~\ref{def:Gamma}.} \label{fig:Lambda_and_Gamma}   
\end{figure}
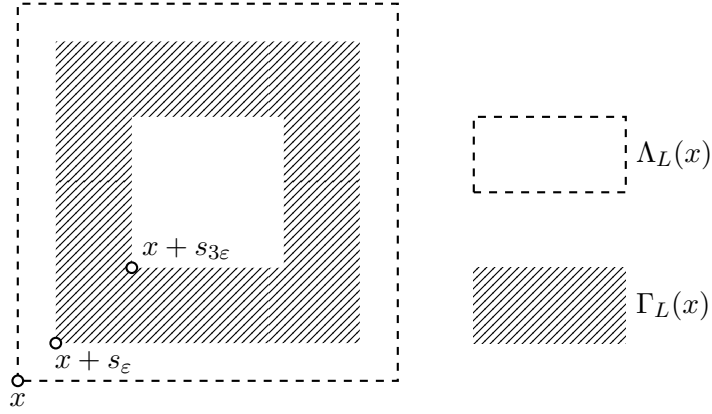

\begin{definition}[Simon-Lieb inequality (SLI)]
    We say that the family of operators $(\cA_\omega^\epsilon)_{\omega \in \Omega}$ satisfies the \emph{Simon-Lieb inequality} in $I_0 \subset \R$ if there exists $\gamma_{I_0} < +\infty$ such that for all $L, \ell', \ell'' \in 2 \epsilon \N$ and $x,y,y' \in \Z^d$ with $\overline{\Lambda_{\ell''}(y)} \subset \Lambda_{\ell'}(y')$ and $\overline{\Lambda_{\ell'}(y')} \subset \Lambda_{L}(x)$, and for all $E \in I_0$ and $\omega \in \Omega$ such that $E \not \in \sigma(\cA_\omega^\epsilon)$ we have
    \begin{align*}
        &\left\lVert
            \chi_{\Gamma_L(x)}
            (\cA_{\omega, L,x}^\epsilon - E)^{-1}
            \chi_{\Lambda_{\ell''(y)}}
        \right\rVert
        \\
        \leq
        &\gamma_{I_0}
\,
        \left\lVert
            \chi_{\Gamma_{\ell'}(y')}
            (\cA_{\omega, \ell',y'}^\epsilon - E)^{-1}
            \chi_{\Lambda_{\ell''(y)}}
        \right\rVert
        \left\lVert
            \chi_{\Gamma_L(x)}
            (\cA_{\omega, L,x}^\epsilon - E)^{-1}
            \chi_{\Lambda_{\ell'(y')}}
        \right\rVert.
    \end{align*}
\end{definition}

The SLI is satisfied for our class or operators, cf.~\cite[Section~2.2.1]{GerminetK-01}.

\begin{definition}[Independence at distance (IAD)]
    We say that the family of operators $(\cA_\omega^\epsilon)_{\omega \in \Omega}$ satisfies \emph{Independence at distance} if there exists $\rho > 0$ such that restrictions boxes with distance larger or equal than $\rho$ are $\PP$-independent.
\end{definition}
It is straightforward to check that this is the case for our family of operators with, say, $\rho=2 \epsilon$.

\begin{definition}[A-priori bound on the number of eigenvalues (NE)]
    Wwe say that the family of operators $(\cA_\omega^\epsilon)_{\omega \in \Omega}$ satisfies an (a priori) \emph{bound on the number of eigenvalues} in $I_0 \subset \R$ if there exists a constant $C_{I_0}>0$ such that for almost all $\omega \in \Omega$, all $x \in \epsilon \Z^d$, and all $L \in \epsilon \N$ we have
    \[
    \#
    \left\{
        j \in \N 
        \colon
        \lambda_j(\cA_{\omega,L,x}^\epsilon) \in  I_0
    \right\}
    \leq
    C_{I_0}\, L^d.
    \]
\end{definition}

We observe that, since $- \epsilon^2 \Delta_{L,x}  \leq \cA_{\omega,L,x}^\epsilon \leq - \Delta_{L,x}$, where $\Delta_{L,x}$ denotes the Laplacian in $L^2(\Lambda_L(x))$ with Dirichlet boundary conditions, property (NE) for our family of operators follows from a straightforward Weyl estimate.

\begin{definition}[Wegner's estimate (W)]
    We say that the family of operators $(\cA_\omega^\epsilon)_{\omega \in \Omega}$ satisfies \emph{Wegner's estimate} in $I_0 \subset \R$ if there exist constants $C_W>0$ and $\kappa \in (0,1)$ such that for all $x \in \epsilon \Z^d$, $L \in \N$, $E \in I_0$, and $\delta > 0$ we have
    \[
    \PP
    \left[
        \dist(\sigma (\cA_{\omega,L,x}^\epsilon, E) \leq \delta
    \right]
    \leq
    C_W
    \,
    \delta^\kappa
    \,
    \lvert \Lambda_L \rvert^{b}.
    \]
\end{definition}

Assumption (W) follows from the Wegner's estimate in Theorem~\ref{thm:Wegner} with $b = 2$ and $\kappa = \frac{1}{C_2}$, see~\cite[Remark~2.3]{GerminetK-01}.

\begin{remark}
Note that in~\cite{GerminetK-01} the Wegner estimate is formulated with $\kappa = 1$.
However, with minor changes, the proof therein applies to the range of parameters $\kappa \in (0,1)$ as well, as per Remark~2.4 therein.
Indeed, retracing the proof of~\cite[Theorem~3.8]{GerminetK-01} (which we recall below as Theorem~\ref{thm:Bootstrap_MSA}) one realises that the parameter $\kappa$ will only influence the finite scale $\mathcal{L}_0(E)$, thus leaving the statement of the theorem unaffected.
\end{remark}

\begin{definition}[Eigenfunction decay inequality (EDI)]
We say that the family of operators $(\cA_\omega^\epsilon)_{\omega \in \Omega}$ satisfies the \emph{Eigenfunction decay inequality} if there exists a constant $\tilde \gamma_{I_0}<+\infty$ such that, given a generalised eigenfunction $\psi$ corresponding to a generalised eigenvalue $E\in I_0$, one has
\begin{equation*}
    \left\|\chi_{\Lambda_1(x)}\psi\right\|\le \tilde \gamma_{I_0} \,
    \left\|\chi_{\Gamma_L(x)}(\cA_{\omega,L,x}^\epsilon - E)^{-1}\chi_{\Lambda_1(x)}\right\|
    \left\|\chi_{\Gamma_L(x)}\psi\right\|
\end{equation*}
for all $x\in \epsilon \Z^d$ and $L\in \N$ such that $E\not \in \sigma(\cA_{\omega,L,x}^\epsilon)$.
\end{definition}

This is satisfied in our case as shown in~\cite{FigotinK-97}.

\begin{definition}[Strong generalised eigenvalue expansion (SGEE)]
We say that $(\cA_\omega^\epsilon)_{\omega \in \Omega}$ satisfies the \emph{strong generalised eigenvalue expansion} if the following holds.
Given $\nu>\frac{d}{4}$, and defining $T$ to be the operator of multiplication by the weight $(1+|x|^2)^{\nu}$ in $L^2(\R^d)$, then the set
\[
\mathcal{D}_\omega^+:=\left\{\psi\in \mathcal{D}(\cA_\omega^\epsilon)\cap L^2(\R^d, (1+|x|^2)^{2\nu} \mathrm{d}x)\ \vline \ \cA_\omega^\epsilon\psi\in L^2(\R^d, (1+|x|^2)^{2\nu} \mathrm{d}x)\right\}
\]
is dense in $L^2(\R^d, (1+|x|^2)^{2\nu} \mathrm{d}x)$ and is an operator core for $\cA_\omega^\epsilon$ almost surely. Furthermore, there exists a bounded continuous function $f$, strictly positive on $\sigma(\cA_\omega^\epsilon)$, such that
\begin{equation*}
    \label{eq:SGEE}
    \E\left[\left(\operatorname{tr}\left[ T^{-1}f\left(\cA_\omega^\epsilon\right)T^{-1} \right]\right)^2\right]< +\infty\,.
\end{equation*}
Here and further on, $\mathrm{tr}$ stands for operator trace.
\end{definition}

The fact that (SGEE) holds in our case follows from the (more general) \cite[Theorem~1.1]{KleinKS-02}.
Let us summarise the discussion so far in the form of a lemma as follows:

\begin{lemma}
\label{lem:multiscale_assumptions}
    For any compact interval $I_0 \subset \R$, the random family of operators $(\cA_{\omega}^\epsilon)_{\omega \in \Omega}$ satisfies (SLI), (IAD), (NE), (W), (EDI), and (SGEE).
\end{lemma}

We are now in a position to cite the following result.

\begin{theorem}[{Bootstrap Multiscale Analysis~\cite[Theorem~3.8]{GerminetK-01}}]
    \label{thm:Bootstrap_MSA}
    Let $\mathcal{I} \subset \R$ be open and $\theta > 2 d$.
    Then, for each $E \in \mathcal{I}$, there exists a finite scale $\mathcal{L}_\theta(E) = \mathcal{L}_\theta(E,d)$, bounded on compact subintervals of $\mathcal{I}$, such that, if for a given $\tilde E \in \Sigma(\mathcal{A}_\omega^\epsilon) \cap \mathcal{I}$ we have 
    \begin{equation*}   
    \label{eq:suitable}
    \PP
    \left[
        \left\lVert
        \chi_{\Gamma_{L_0}} (\cA_{\omega,L_0}^\epsilon - \tilde E)^{-1} \chi_{\Lambda_{\frac{L_0}{3}}}
        \right\rVert
        \leq
        \frac{1}{L_0^\theta}
    \right]
    >
    1
    -
    \frac{1}{841^d}
    \end{equation*}
for some $L_0 \in 6\epsilon\N$, $L_0 > \mathcal{L}_\theta(E)$, then there exists $\delta_0 > 0$ such that, defining $I(\delta_0) := [\tilde E - \delta_0, \tilde E + \delta_0]$, 
we find for every $\zeta \in (0,1)$ a constant $C_\zeta > 0$ such that
	\begin{equation*}
		\E
		\left[
			\sup_{\lVert f \rVert_\infty \leq 1}
			\left\lVert
				\chi_{\Lambda_1(x)} 
				\chi_{I(\delta_0)}(\cA_\omega^\epsilon) 
                f(\cA_\omega^\epsilon) 
				\chi_{\Lambda_1(y)}
			\right\rVert_{\mathrm{HS}}^2
		\right]
		\leq
		C_\zeta\,
		\mathrm{e}^{- \lvert x - y \rvert^\zeta}
	\end{equation*}
	for all $x,y \in \epsilon \Z^d$. The supremum is taken over all Borel-measurable functions and $\lVert\, \cdot\, \rVert_{\mathrm{HS}}$ denotes the Hilbert-Schmidt norm.
\end{theorem}

Using compactness of the interval $\left[E_0, E_0 + \nicefrac{L_0^{-\nicefrac{1}{2}}}{2}\right]$ for $L_0 > 0$ and Lemma~\ref{lem:multiscale_assumptions}, we deduce:

\begin{corollary}
    \label{cor:Bootstrap_MSA}
    Let $E_0 \in \R$ and $\theta > 2 d$.
    If 
    \begin{equation}  
    \label{eq:limsup_suitability}
    \limsup_{L \to \infty}
    \
    \PP
    \left[
        \left\lVert
        \chi_{\Gamma_{L}} (\cA_{\omega,L}^\epsilon - E)^{-1} \chi_{\Lambda_{\frac{L}{3}}}
        \right\rVert
        \leq
        \frac{1}{L^\theta}
        \quad
        \text{for all}
        \quad
        E \in 
        \left[
        E_0, E_0 + \nicefrac{L_0^{-\nicefrac{1}{2}}}{2}
        \right]
    \right]
    =
    1\,,
    \end{equation}
    then there exists $L_0 \in \epsilon \N$ such that, defining $I_0 := \left[E_0, E_0 + \nicefrac{L_0^{-\nicefrac{1}{2}}}{2}\right]$, one can find for every $\zeta \in (0,1)$ a constant $C_\zeta > 0$ such that
	\begin{equation*}
		\E
		\left[
			\sup_{\lVert f \rVert_\infty \leq 1}
			\left\lVert
				\chi_{\Lambda_1(x)} 
				\chi_{I_0}(\cA_\omega^\epsilon) 
                f(\cA_\omega^\epsilon) 
				\chi_{\Lambda_1(y)}
			\right\rVert_{\mathrm{HS}}^2
		\right]
		\leq
		C_\zeta\,
		\mathrm{e}^{- \lvert x - y \rvert^\zeta}
	\end{equation*}
	for all $x,y \in \epsilon \Z^d$. The supremum is taken over all Borel-measurable functions and $\lVert\, \cdot\, \rVert_{\mathrm{HS}}$ denotes the Hilbert-Schmidt norm.
\end{corollary}

Let us emphasise that Theorem~\ref{thm:Bootstrap_MSA} and Corollary~\ref{cor:Bootstrap_MSA} rely on the knowledge that the assumptions listed in Lemma~\ref{lem:multiscale_assumptions} are fulfilled for the random family of operators $(\mathcal{A}_\omega^\epsilon)_{\omega \in \Omega}$.

\

    The proof of Theorem~\ref{thm:HS_kernel_decay} (and thus also of Theorems~\ref{thm:Anderson_localization} and~\ref{thm:Dynamical_localization}) is complete as soon as we have the following.

\begin{theorem}
    \label{thm:suitability_holds}
    Let $E_0 \in \R$ be a lower band edge of $\sigma(\mathcal{A}_\omega)$.
    Then~\eqref{eq:limsup_suitability} holds.
\end{theorem}    

The proof of Theorem~\ref{thm:suitability_holds} relies, in turn, on the initial scale estimate --- i.e., Theorem~\ref{thm:ISE_strong} --- in combination with the following Combes--Thomas estimate.

\begin{proposition}[{Combes--Thomas estimate~\cite[Theorem 2.1]{GerminetK-03}}]
    \label{prop:Combes_thomas}
    There exists a constant $C_{\mathrm{CT}} > 0$ depending only on the dimension such that, for all $L \in \epsilon \N$, all $x,y \in \epsilon \Z^d$, and all $E \in \R \backslash \sigma(\cA_{\omega,L}^\epsilon)$ we have
    \begin{equation*}
        \lVert
        \chi_{\Lambda_\epsilon(x)}
        (\cA_{\omega,L}^\epsilon - E)^{-1}
        \chi_{\Lambda_\epsilon(y)}
        \rVert
        \leq
        C_{\mathrm{CT}}
        \,
        \exp
        \left(
            - \frac{\dist (\sigma(\cA_{\omega, L}^\epsilon), E )}{C_{\mathrm{CT}}}\,
            \lvert x - y \rvert
        \right).
    \end{equation*}
\end{proposition}

\begin{proof}[{Proof of Theorem~\ref{thm:suitability_holds}}]
   Proposition~\ref{prop:Combes_thomas} implies
    \begin{multline}
      \label{eq:ISE_for_MSA_0}
        \left\lVert
            \chi_{\Gamma_L} (\cA_{\omega,L}^\epsilon - E)^{-1} \chi_{\Lambda_{\frac{L}{3}}}
       \right\rVert
        \leq
        \sum_{j \in \Gamma_L \cap \epsilon \Z^d}
        \sum_{k \in \Gamma_{\frac{L}{3}} \cap \epsilon \Z^d}
        \lVert
            \chi_{\Lambda_\epsilon(j)} (\cA_{\omega,L}^\epsilon - E)^{-1} \chi_{\Lambda_\epsilon(k)}
        \rVert
        \\
        \leq
        \frac{L^{2d}}{\epsilon^{2d}}
        C_{\mathrm{CT}}
        \exp
        \left(
            - \frac{L}{4 C_{\mathrm{CT}}} \dist (\sigma(\cA_{\omega, L}^\epsilon), E )
        \right).
    \end{multline}
    Since $E_0$ is a lower band edge, one can find almost surely a sufficiently large $L$ such that 
    \begin{equation}
    \label{eq:ISE_for_MSA}
    \sigma(\cA_{\omega,L}) \cap [E_0, E_0 + L^{-\nicefrac{1}{2}}]
    =
    \emptyset\,,
    \end{equation}
   that is,
     \begin{equation}
    \label{eq:ISE_for_MSA_2}
    \dist(\sigma(\cA_{\omega,L}), E)
    \geq
    \frac{L^{-\nicefrac{1}{2}}}{2}
    \quad
    \text{for all}
    \quad
    E \in 
        \left[
        E_0, E_0 + \frac{L^{-\nicefrac{1}{2}}}{2}
        \right].
        \end{equation}
    Combining formulae~\eqref{eq:ISE_for_MSA_0} and~\eqref{eq:ISE_for_MSA_2} one obtains for sufficiently large $L$ the estimate
    \[
    \left\lVert
            \chi_{\Gamma_L} (\cA_{\omega,L}^\epsilon - E)^{-1} \chi_{\Lambda_{\frac{L}{3}}}
        \right\rVert
    \leq
    \frac{L^{2d}}{\epsilon^{2d}}
    C_{\mathrm{CT}}
        \exp
        \left(
            - \frac{L^{\frac{1}{2}}}{8 C_{\mathrm{CT}}}
        \right)
        \leq
        \frac{1}{L^\theta} \quad
    \text{for all}
    \quad
    E \in 
        \left[
        E_0, E_0 + \frac{L^{-\nicefrac{1}{2}}}{2}
        \right].
    \]
    Now, by Theorem~\ref{thm:ISE_strong} for any $C_4 > 0$ and sufficiently large $L$ the event~\eqref{eq:ISE_for_MSA} has probability at least $1 - L^{-C_4}$.
    But the latter tends to $1$ as $L\to +\infty$, and this concludes the proof.
\end{proof}

\section*{Acknowledgements}
\addcontentsline{toc}{section}{Acknowledgements}

\

The authors would like to thank Igor Vel\v{c}i\'c and Ivan Veseli\'{c} for stimulating discussions. 
MC is grateful to the 
Department of Mathematics at Yale University and CERAMATHS at the Université Polytechnique Hauts-de-France, where part of this work has been done, for the kind and welcoming hospitality. MT would like to thank the Department of Mathematics at Heriot-Watt University, where this work has been initiated.

\

\noindent\emph{Funding}.\\
MC was partially supported by EPSRC Fellowship EP/X01021X/1 and by a travel grant from UPHF. Both are gratefully acknowledged.
The research of MT was supported by Agence Nationale de Recherche under the grant \emph{Modélisation mathématique et optimisation de la Fabrication Additive céramique} (Chaire de Professeur Junior).


\newcommand{\etalchar}[1]{$^{#1}$}

\end{document}